\documentclass[platex,dvipdfmx]{amsart}
\usepackage{amssymb}
\usepackage{amscd}
\usepackage{braket}
\usepackage{mathrsfs}
\usepackage{amsrefs}
\usepackage{graphicx}

\usepackage{comment} 
\usepackage{color} 
\usepackage{mleftright}
\usepackage{hyperref}
\usepackage{cleveref}
\usepackage[all]{xy}

\theoremstyle{plain}
\newtheorem{theo}{Theorem}[section]
\crefname{theo}{Theorem}{Theorems}
\Crefname{theo}{Theorem}{Theorems}
\newtheorem{prop}[theo]{Proposition}
\crefname{prop}{Proposition}{Propositions}
\Crefname{prop}{Proposition}{Propositions}
\newtheorem{lem}[theo]{Lemma}
\crefname{lem}{Lemma}{Lemmas}
\Crefname{lem}{Lemma}{Lemmas}
\newtheorem{cor}[theo]{Corollary}
\crefname{cor}{Corollary}{Corollaries}
\Crefname{cor}{Corollary}{Corollaries}

\crefname{claim}{Claim}{Claims}
\Crefname{claim}{Claim}{Claims}

\crefname{property}{Property}{Properties}
\Crefname{property}{Property}{Properties}

\crefname{problem}{Problem}{Problems}
\Crefname{problem}{Problem}{Problems}

\theoremstyle{definition}
\newtheorem{defi}[theo]{Definition}
\crefname{defi}{Definition}{Definitions}
\Crefname{defi}{Definition}{Definitions}

\crefname{notation}{Notation}{Notations}
\Crefname{notation}{Notation}{Notations}

\crefname{convention}{Convention}{Conventions}
\Crefname{convention}{Convention}{Conventions}

\crefname{cond}{Condition}{Conditions}
\Crefname{cond}{Condition}{Conditions}
\newtheorem{conj}[theo]{Conjecture}
\crefname{conj}{Conjecture}{Conjecture}
\Crefname{conj}{Conjecture}{Conjecture}

\crefname{assum}{Assumption}{Assumptions}
\Crefname{assum}{Assumption}{Assumptions}

\theoremstyle{remark}
\newtheorem{rem}[theo]{Remark}
\crefname{rem}{Remark}{Remarks}
\Crefname{rem}{Remark}{Remarks}

\crefname{ex}{Example}{Examples}
\Crefname{ex}{Example}{Examples}

\crefname{section}{Section}{Sections}
\Crefname{section}{Section}{Sections}
\crefname{subsection}{Subsection}{Subsections}
\Crefname{subsection}{Subsection}{Subsections}
\crefname{figure}{Figure}{Figures}
\Crefname{figure}{Figure}{Figures}

\newcommand{\Z}{\mathbb{Z}}

\newcommand{\R}{\mathbb{R}}
\newcommand{\C}{\mathbb{C}}
\renewcommand{\H}{\mathbb{H}}
\newcommand{\Q}{\mathbb{Q}}
\newcommand{\CP}{\mathbb{CP}}

\newcommand{\sign}{\mathrm{sign}}

\newcommand{\calB}{{\mathcal B}}

\newcommand{\fraks}{\mathfrak{s}}

\newcommand{\ind}{\mathop{\mathrm{ind}}\nolimits}
\newcommand{\Ind}{\mathop{\mathrm{Ind}}\nolimits}

\newcommand{\Spinc}{spin$^{c}$ }

\newcommand{\UU}{\operatorname{\rm U}}

\newcommand{\BF}{\mathrm{BF}}
\newcommand{\SW}{\mathrm{SW}}
\newcommand{\SWF}{\mathrm{SWF}}

\newcommand{\BFdim}{\dim_{\mathrm{BF}}}

\numberwithin{equation}{section}


\date{\today}
\title[Notions of simple type for Bauer--Furuta invariants]{Notions of simple type for Bauer--Furuta invariants}

\author{Tsuyoshi Kato}
\address{Department of Mathematics, Kyoto University, Kyoto, 606-8502, Japan}
\email{tkato@math.kyoto-u.ac.jp}

\author{Daisuke Kishimoto}
\address{Faculty of Mathematics, Kyushu University, Fukuoka 819-0395, Japan}
\email{kishimoto@math.kyushu-u.ac.jp}

\author{Nobuhiro Nakamura}
\address{Center for Integrated Sciences and Humanities, Fukushima Medical University, 
1 Hikariga-oka, Fukushima City 960-1295, Japan}
\email{nnaka@fmu.ac.jp}

\author{Kouichi Yasui}
\address{Department of Pure and Applied Mathematics, Graduate School of Information Science and Technology, The University of Osaka, 1-5 Yamadaoka, Suita, Osaka 565-0871, Japan}
\email{kyasui@ist.osaka-u.ac.jp}

\begin{document}

\maketitle

\begin{abstract}

By extending the notion of simple type for the Seiberg--Witten invariant of a 4-manifold, we introduce notions of BF blowup simple type and BF homogeneous type for the Bauer--Furuta invariant and study their applications. Specifically, we show that the existence of an immersed 2-sphere with a certain condition guarantees BF blowup simple type. As an application, we determine the Bauer--Furuta invariant of a 4-manifold obtained by a logarithmic transformation along a torus in a fishtail neighborhood. We also give constraints on gluing decompositions of 4-manifolds by using BF homogeneous type.  To prove these results, we also give gluing formulae and an immersed adjunction inequality for Bauer--Furuta invariants. 
\end{abstract}

\section{Introduction}\label{sec:introduction}

The Seiberg-Witten invariants \cite{Witten} of smooth $4$-manifolds have played a pivotal role in the study of $4$-manifolds over the past thirty years, yielding numerous significant applications to low-dimensional topology.
The simple type conjecture, posed in 1990s, states a fundamental constraint on the invariant (e.g. \cite[\S2.3.3]{Nicolaescu}).

\begin{conj}[Simple type conjecture]
 Every closed, connected, oriented, smooth
4-manifold with $b_2^+ >1$ has Seiberg--Witten simple type.
\end{conj}
Here a closed, connected, oriented, smooth 4-manifold $X$ with $b_2^+ >1$ is called of
Seiberg--Witten simple type if the (integer valued) Seiberg--Witten invariant $\SW(X,\fraks)$ of a \Spinc structure $\fraks$ on $X$ is zero whenever the virtual dimension $d(\fraks)$ of the Seiberg--Witten moduli space for $\fraks$ 
is non-zero.

The Bauer--Furuta invariant \cite{BF,BF2,B_survey} is a stable cohomotopy refinement of the Seiberg--Witten invariant.
Bauer \cite{BF2} proved that the Bauer--Furuta invariant is  strictly stronger than the Seiberg--Witten invariant by
showing a class of 4-manifolds, e.g., connected-sums of several $K3$ surfaces, has the
property that the Seiberg--Witten invariant is trivial but the Bauer--Furuta invariant is not.

In this paper, we introduce two notions, BF blowup simple type and BF homogeneous type,   for the Bauer--Furuta invariant, extending the notion of  Seiberg--Witten simple type.
The former notion extends a property of a $4$-manifold of Seiberg-Witten simple type with regard to the blowup formula.
The latter notion naturally extends the definition of Seiberg--Witten simple type in terms of virtual dimensions of the moduli spaces.
Furthermore, we give their applications to $4$-dimensional topology.


\subsection{Generalized blowup formula}\label{subsec:blowup} 
In order to define the notion of blowup simple type, we present a generalized form of blowup formula for Bauer--Furuta invariants. 
Throughout the paper, we assume every manifold is  connected, oriented and smooth unless stated otherwise.

Let $Y$ be a closed  $3$-manifold with $b_1(Y)=0$ which admits a positive scalar curvature metric.
Let $X$, $N_1$, $N_2$ be  compact oriented $4$-manifolds such that $\partial X=Y=-\partial N_i$ for $i=1,2$. Note that we can treat the case $N_1=N_2$ as well. 
Suppose   $b_1(N_i)=b_2^+(N_i)=0$ for $i=1,2$.
Let $X_i = X\cup_{Y}N_i$.
Then $b_1(X_1)=b_1(X_2)$ and $b_2^+(X_1)=b_2^+(X_2)$. For example, we can apply this setting for a closed $4$-manifold $X_1$ and its blowup $X_2=X_1\#\overline{\CP}^2$.

Let $\fraks_i$ be a \Spinc structure on $X_i$. Suppose $\fraks_1|_X = \fraks_2|_X$. 
The virtual dimension of the Seiberg--Witten moduli space for $(X_i,\fraks_i)$ is given by
\begin{equation}\label{eq:SWdim}
\begin{aligned}
d(\fraks_i) &:= \frac14(c_1(L_i)^2-\sigma(X_i))-(1-b_1(X_i)+b_2^+(X_i))\\&=\frac14(c_1(L_i)^2-2\chi(X_i)-3\sigma(X_i)),
\end{aligned}
\end{equation}
where $L_i$ is the determinant line bundle of $\fraks_i$, and $\chi(X_i)$ and $\sigma(X_i)$ respectively denote the Euler characteristic and the signature of $X_i$.
For a line bundle $L$, the above quantity for $L$ is denoted by $d(L)$.

The Bauer--Furuta invariant $\BF(X_i,\fraks_i)$ of $(X_i,\fraks_i)$  is defined as a class of the $\UU(1)$-stable cohomotopy group $\pi^0_{S^1}(Pic(X_i);\lambda_i)$.
(See \cite{BF, B_survey} and \cref{subsec:BF} for the definition.) 
First we study the relation between $\BF(X_1,\fraks_1)$ and $\BF(X_2,\fraks_2)$.

We will see that, if $\delta=d(\fraks_2)- d(\fraks_1)\geq 0$, then  there is a homomorphism $I_\delta\colon\pi^0_{S^1}(Pic(X_2);\lambda_2)\to \pi^0_{S^1}(Pic(X_1);\lambda_1)$, and the following formula holds (\cref{thm:gblowup}):
\[
\BF(X_1,\fraks_1)=I_{\delta} \BF(X_2,\fraks_2).
\]
This formula can be considered as a generalization of the blowup formula by Bauer \cite[Corollay 4.2]{BF2}, and implies the next theorem.
\begin{theo}\label{thm:negdef}\ 
\begin{enumerate}
\item If $d(\fraks_1)\leq d(\fraks_2)$ and $\BF(X_1,\fraks_1)\neq 0$, then $\BF(X_2,\fraks_2)\neq 0$.
\item If $d(\fraks_1) = d(\fraks_2)$, then $\BF(X_1,\fraks_1)=\BF(X_2,\fraks_2)$.
\end{enumerate}
\end{theo}

For a closed oriented $4$-manifold $Z$ with   a \Spinc structure $\fraks$ on $Z$, we will consider the following condition.
\begin{center}
Condition $(\ast)$: $b_1(Z)=0$, $b_2^+(Z)\geq 2$ and $d(\fraks)\leq3$.  If $d(\fraks)>0$, then $\BF(Z,\fraks)$ is a torsion element. 
\end{center}
To our knowledge, all known examples of $(Z,\fraks)$ such that  $b_1(Z)=0$, $b_2^+(Z)\geq 2$ and $\BF(Z,\fraks)\neq 0$ satisfy Condition ($\ast$).
Such an example of  $(Z,\fraks)$ with $\BF(Z,\fraks)\neq 0$ and $d(\fraks)>0$ is obtained by taking a connected sum of several $4$-manifolds with non-zero Seiberg-Witten invariants.
See \cref{thm:BF-nonvan}.

\begin{rem}
Suppose $b_1(Z)=0$ and $b_2^+(Z)\geq 2$.
If $d(\fraks)=0$, then $\BF(Z,\fraks)$ is identified with the Seiberg--Witten invariant $\SW(Z,\fraks)$ by \cref{thm:BF-SW} (\cite[Theorem 1.1]{BF}). 
If $d(\fraks)$ is odd, then   $\BF(Z,\fraks)$ is a torsion element.
If $Z$  is of Seiberg--Witten simple type and $d(\fraks)>0$, then $\BF(Z,\fraks)$ is a torsion element  by \cref{thm:BF-SW}.
\end{rem}

\begin{theo}\label{thm:vanishing}
Suppose $(X_2,\fraks_2)$ satisfies Condition $(\ast)$.
If $d(\fraks_1)< d(\fraks_2)$ then $\BF(X_1,\fraks_1) = 0$.
\end{theo}
\begin{rem}
In \cref{thm:negdef}, it can happen that $d(\fraks_1)< d(\fraks_2)$, $\BF(X_1,\fraks_1)\neq 0$ and $\BF(X_2,\fraks_2) \neq0$.
\cref{thm:vanishing} says that such a situation does not occur  if $(X_2,\fraks_2)$ satisfies Condition $(\ast)$.
\end{rem}

We have two corollaries of \cref{thm:negdef} and \cref{thm:vanishing}.
The first one is a blowup formula.
Let $E\in H^2(\overline{\CP}^2;\Z)$ be the Poincar\'{e} dual of the exceptional sphere  of $\overline{\CP}^2$.
For each integer $r$, there is a \Spinc structure $\fraks_{(2r+1)E}$ on $\overline{\CP}^2$ which is unique up to isomorphism such that $c_1(\fraks_{(2r+1)E})=(2r+1)E$.
Let $X$ be a closed oriented smooth $4$-manifold
with a \Spinc structure  $\fraks$.
For each integer $r$, consider  the \Spinc structure $\fraks_r=\fraks\#\fraks_{(2r+1)E}$ on $X\#\overline{\CP}^2$.
Then we have
\[
d(\fraks_r)=d(\fraks)-r(r+1).
\]
\begin{theo}[\textit{Cf. \cite[Corollary 4.2]{BF2}}]\label{thm:blowup}
If $\BF(X\#\overline{\CP}^2,\fraks_r)\neq 0$, then $\BF(X,\fraks)\neq 0$.
Moreover
\[
\BF(X,\fraks)=\BF(X\#\overline{\CP}^2,\fraks_0) = \BF(X\#\overline{\CP}^2,\fraks_{-1}). 
\]
Suppose further that $(X,\fraks)$ satisfies Condition $(\ast)$. Then
\begin{equation}\label{eq:BFup}
\BF(X\#\overline{\CP}^2,\fraks_r) = 0\text{ if $r\neq 0,-1$.} 
\end{equation}
\end{theo}

The second one is a formula for rational blowdown.
Let $C_p$ and $B_p$ be the $4$-manifolds defined  by Fintushel--Stern~\cite{FSrb}.
The manifold $C_p$ is obtained by the plumbing of certain $(p-1)$ disk bundles over $2$-spheres, and $B_p$ is a rational homology $4$-ball such that $\partial B_p=\partial C_p$ and $C_p\cup(-B_p)=(p-1)\overline{\CP}^2$.
Suppose $C_p$ is embedded in a closed oriented smooth $4$-manifold $X$.
Then $X_p=(X\setminus \overset{\circ}{ C_p})\cup B_p$ is called a rational blowdown \cite{FSrb}.
Let $\fraks_p$ be a \Spinc structure on $X_p$.
A \Spinc structure $\fraks$ on $X$ is called a {\it lift of $\fraks_p$} if 
 $\fraks|_{(X\setminus \overset{\circ}{ C_p})} = \fraks_p|_{(X_p\setminus \overset{\circ}{ B_p})}$.
The fact that $d(\fraks)\leq d(\fraks_p)$ is proved in \cite{FSrb}.
\begin{theo}\label{thm:rbd}
Let $\fraks$ be a lift of $\fraks_p$.
\begin{enumerate}
\item If $\BF(X,\fraks)\neq 0$, then $\BF(X_p,\fraks_p)\neq 0$.
\item If $d(\fraks) = d(\fraks_p)$, then $\BF(X,\fraks)=\BF(X_p,\fraks_p)$.
\item Suppose that $(X_p,\fraks_p)$ satisfies Condition $(\ast)$.
If $d(\fraks) < d(\fraks_p)$, then $\BF(X,\fraks)=0$.
\end{enumerate}
\end{theo}
It is also proved in \cite{FSrb} that each characteristic line bundle $\bar{L}$ on $X_p$ has  a lift $L$ on $X$ which is also characteristic and satisfies $d(L)=d(\bar{L})$.
From this, the Bauer--Furuta invariants of $X_p$ can be determined by those of $X$ in some cases.
For example, we obtain the following corollary.
\begin{cor}\label{cor:rbd}
Suppose that $H_1(X;\Z)$ and $H_1(X_p;\Z)$ have no $2$-torsion element.
For each \Spinc structure $\fraks_p$ on $X_p$, there exists a unique lift $\fraks$ on $X$ of $\fraks_p$ such that $d(\fraks_p)=d(\fraks)$, and therefore   $\BF(X_p,\fraks_p)=\BF(X,\fraks)$.
\end{cor}

\subsection{Immersed adjunction inequality and BF blowup simple type} 
We here give the definition of Bauer--Furuta basic class.
\begin{defi}
For a closed oriented $4$-manifold $X$, a second cohomology class $c\in H^2(X;\Z)$  is called  a \textit{Bauer--Furuta basic class}, if there exists a \Spinc structure $\fraks$ on $X$ such that $\BF(X,\fraks)\neq 0$ and $c_1(\fraks)=c$. 
Let $\calB(X)$  be the set of Bauer--Furuta basic  classes of $X$.
\end{defi}

The following theorem is a variant of the immersed adjunction inequlity by Fintushel--Stern \cite[Theorem 1.3]{FSimmersed}. 
\begin{theo}\label{thm:immersedadjunction}
Let $X$ be a closed oriented smooth $4$-manifold. 
Let $K$ be a Bauer--Furuta basic class.
If $\alpha\in H_2(X;\Z)$ is represented by an immersed sphere with $p$ positive double points, then either 
\[
2p-2 \geq  |\langle K,\alpha \rangle| +\alpha\cdot\alpha,
\]
or
\[
\left\{\begin{aligned}
\BF(X,\fraks+\alpha^*)\neq0\ &\quad \text{if }\langle K,\alpha \rangle\geq 0\\
\BF(X,\fraks-\alpha^*)\neq0\ &\quad \text{if }\langle K,\alpha \rangle\leq 0,
\end{aligned}\right.
\]
where $\alpha^*$ is the Poincar\'{e} dual of $\alpha$.
\end{theo}
\begin{rem}
\cref{thm:immersedadjunction} formally holds  for $4$-manifolds with arbitrary $b_2^+$. 
However, this does not mean that we can obtain a nontrivial estimate  of $p$ when $b_2^+=0$.
In fact, if $b_2^+=0$, then the condition that $\BF(X,\fraks\pm\alpha^*)\neq0$ is always true (\cref{rem:b+0}), and therefore  no estimate on $p$ is obtained.
\end{rem}

Inspired by \cref{thm:blowup}, we introduce the notion of the BF blowup simple type (\cref{Y:def:blowup simple type}) which says that \eqref{eq:BFup} holds for every \Spinc structure.
This notion can be seen as a generalization of the Seiberg--Witten simple type.
By using the adjunction inequality for embedded surfaces (\cref{cor:adj}), we obtain the following proposition which relates the existence of embedded surface with the property of BF blowup simple type.

\begin{prop}\label{Y:claim:simple:embedded} If $X$  admits a smoothly embedded closed oriented surface of genus $g>1$ with self-intersection number $2g-2$, then $X$ is of BF blowup simple type. 
\end{prop}

\cref{thm:immersedadjunction} and \cref{Y:claim:simple:embedded} are used to prove
the next theorem which treats a similar sufficient condition for BF blowup simple type in terms of immersed spheres,  that includes the case where the self-intersection number is zero.

\begin{theo}\label{Y:claim:simple:immersed} { Suppose $X$ admits an immersed $2$-sphere with $p$ positive double points and possibly with negative double points, that represents a non-torsion second homology class with self-intersection number $2p-2\geq 0$.  Then, $X$ is of BF blowup simple type.} 
\end{theo}

\subsection{Formula for logarithmic transformation} 
As an application of \cref{cor:rbd} and \cref{Y:claim:simple:immersed}, we determine the Bauer--Furuta invariant for a logarithmic transformation along a torus embedded in a fishtail neighborhood. 
We remark that such logarithmic transformations have varieties of applications to 4-dimensional topology as is well-known.
Let $X$ be a connected closed oriented $4$-manifold.
Suppose that $X$ contains an immersed $2$-sphere $F$ with exactly one double point such that  its homology class $\alpha=[F] \in H_2(X;\Z)$ is a non-torsion class and satisfies $\alpha\cdot\alpha=0$.
Then a neighborhood  $Q$ of $F$ is a fishtail neighborhood.
Performing a logarithmic transformation with multiplicity $p$ (and some auxiliary data given in \cite[Theorem~8.5.9 and Remark~8.5.10(b)]{GS}) along the torus in the fishtail neighborhood $Q$, we obtain a $4$-manifold $X_{(p)}$. Note that the diffeomorphism type of $X_{(p)}$ is uniqulely determined by the multiplicity $p$ in the case where the torus is a regular fiber of a cusp neighborhood (see \cite{GS}). 
Let $\beta\in H_2(X_{(p)};\Z)$ be the homology class of the multiple fiber and $f\in H^2(X_{(p)};\Z)$ be the Poincar\'{e} dual of $\beta$. 
Note that  $\alpha = p\beta$ holds.
\begin{theo}\label{thm:log}
Suppose that $H_1(X;\Z)$ and $H_1(X_{(p)};\Z)$ have no $2$-torsion element.
Then the set of basic classes of $X_{(p)}$ is given by
\[
\calB(X_{(p)}) = \left\{L_k=L + \left.(2k-(p-1)) f \,\right|\, L\in \calB(X),  k\in\{0,1,\ldots, p-1\}\right\},
\]
and the Bauer--Furuta invariants satisfy the relation
\[
\BF(X,\mathfrak{s}_L)=\BF(X_{(p)}, \mathfrak{s}_{L_k})
\]
where $\mathfrak{s}_L$ (resp. $\mathfrak{s}_{L_k}$) is a unique  \Spinc structure corresponding to the class $L$ (resp. $L_k$).
\end{theo}
\begin{rem}
Khandhawit--Lin--Sasahira~\cite{KLS} gave a gluing formula of Bauer--Furuta invariants for 
4-manifolds obtained by gluing along 3-manifolds with $b_1 > 0$.
But it is not clear for us whether their formula can completely determine the Bauer--Furuta invariant of $X_{(p)}$.
\end{rem}

\subsection{Constraints on gluing decompositions and BF homogeneous type}
%
%
We discuss  connected-sum decompositions of $4$-manifolds and their generalizations.
Below we consider compact $4$-manifolds with boundary of special kind.

We call a closed connected $3$-manifold $Y$ with $b_1(Y)=0$  \textit{SWF-spherical} for a \Spinc structure $\fraks_Y$ on $Y$ if the Seiberg--Witten Floer homotopy type of $(Y,\fraks_Y)$ has the simplest form, that is, it is represented by a sphere  (\cref{def:SWF-spherical}).
We say the (possibly disconnected) boundary  $Y=\partial X$ of a compact $4$-manifold $X$ is  of \textit{SWF-spherical type} if $b_1(Y)=0$ and each component of $Y$ is SWF-spherical for every \Spinc structure on the component (\cref{def:spherical}).
If a closed $3$-manifold $Y$ with $b_1(Y)=0$ admits a positive scalar curvature metric, then $Y$ is SWF-spherical for every \Spinc structure on $Y$ \cite{ManolescuSWF, Manolescu-gluing}.
It is known that a closed 3-manifold admits a positive scalar curvature metric if and only if it is a connected sum of spherical space forms and copies of $S^1\times S^2$. (See \cite[Theorem 1.29]{Lee}.) 

In the viewpoint of gauge theory, a compact $4$-manifold $X$ with boundary of SWF-spherical type looks like a closed $4$-manifold.
While the relative Seiberg--Witten invariant is in general defined as a class of the monopole Floer homology group, that of $X$ with boundary of SWF-spherical type can be assumed $\Z$-valued. 
Moreover, if every boundary component is $S^3$, then the relative Seiberg--Witten invariants of $X$ coincide with the Seiberg--Witten invariants of the closed manifold obtained by gluing $4$-balls to $X$ along $3$-spheres of the boundary.

\begin{rem}
The previous theorems (from \cref{thm:negdef} to \cref{thm:log}) are generalized to similar results for $4$-manifolds with boundary of SWF-spherical type.
\end{rem}

Let $X$ be a  compact connected $4$-manifold possibly with boundary which is not necessary connected. 
Let  $X_1,\cdots,X_m$  be compact  $4$-manifolds with boundary of SWF-spherical type. 
Each manifold $X_i$ may have two or more boundary components.
We assume these manifolds are glued successively as follows.
Let $\tilde{X}_1=X_1$. 
For $\tilde{X}_i$,  suppose  there is an orientation-reversing diffeomorphism between a boundary component $Y_i$ of $\tilde{X}_i$  and a boundary component $Y_{i+1}$ of $X_{i+1}$.  
Let $\tilde{X}_{i+1}$ be the manifold obtained by gluing $X_{i+1}$ to $\tilde{X}_i$ along $Y_i \cong -Y_{i+1}$.
Finally we obtain a glued connected $4$-manifold $\tilde{X}_m$.
The manifold $\tilde{X}_m$ is denoted by $X_1\cup_{Y_1}X_2\cup_{Y_2}\cdots\cup_{Y_{m-1}} X_m$.
When a $4$-manifold $X$ is diffeomorphic to $X_1\cup_{Y_1}X_2\cup_{Y_2}\cdots\cup_{Y_{m-1}} X_m$, we call it a \textit{gluing decomposition along rational homology $3$-spheres of SWF-spherical type}.
If $m=1$, then we assume $X=X_1$.

It is well-known that a symplectic $4$-manifold with $b_2^+\geq 2$ has no connected sum decomposition into two or more symplectic $4$-manifolds.
The next theorem is a generalization of this fact.
This is proved by using
 \cref{thm:BF-nonvan} stated below which is a direct extension of Bauer's non-vanishing result \cite[Proposition 4.5]{BF2}.
 \begin{theo}\label{thm:decomp}
Suppose $X=X_1\cup_{Y_1}\cdots\cup_{Y_{k-1}}X_k$ $(1\leq k\leq 4)$ is a gluing decomposition along rational homology $3$-spheres of SWF-spherical type with a \Spinc structure $\fraks$ satisfying the following conditions:
\begin{itemize}
\item $b_1(X)=0$, $b_2^+(X_i)\equiv 3$ mod $4$ for each $i$.
\item $d(\fraks_i)=0$ and $\SW(X_i,\fraks_i)$ is odd for each $i$, where $\fraks_i=\fraks|_{X_i}$.
\item If $k=4$ then $b_2^+(X)\equiv 4$ mod $8$.
\end{itemize}
If $X$  admits another gluing decomposition $X^\prime_1\cup_{Y^\prime_1}\cdots\cup_{Y^\prime_{l-1}} X^\prime_{l}$ along rational homology $3$-spheres of SWF-spherical type with $b_2^+(X_j^\prime)\geq 1$ ($j=1,\ldots,l$), then $k\geq l$ and $b_2^+(X^\prime_j)\geq 2$ for every $j$.
\end{theo}

In \cref{sec:BFdim}, we introduce the notion of \textit{BF homogeneous} type which is another generalization of  the Seiberg--Witten simple type.
A closed  $4$-manifold or a compact  $4$-manifold $X$ with boundary of SWF-spherical type with $b_2^+(X)\geq 1$ is of \textit{$d$-dimensional BF homogeneous type} if  $\BF(X,\fraks)=0$ whenever $d(\fraks)\neq d$ (\cref{def:BFhom}).

The connected sum of symplectic $4$-manifolds with $b_2^+\geq 2$ is of BF homogeneous type (\cref{cor:sympconn}).
When the components of a gluing decomposition along rational homology $3$-spheres of SWF-spherical type are of $0$-dimensional BF homogeneous type, \cref{thm:decomp} is refined to the next theorem.
\begin{theo}\label{thm:0decomp}
Suppose $X=X_1\cup_{Y_1}\cdots\cup_{Y_{k-1}}X_k$ ($1\leq k \leq 4$) is a gluing decomposition along rational homology $3$-spheres of SWF-spherical type  with a \Spinc structure $\fraks$ satisfying the following conditions:
\begin{enumerate}
\item $b_1(X_i)=0$, $b_2^+(X_i)\equiv 3$ mod $4$ for each $i$.
\item $d(\fraks_i)=0$ and $\SW(X_i,\fraks_i)$ is odd for each $i$, where $\fraks_i=\fraks|_{X_i}$.
\item $X_i$ is of $0$-dimensional BF homogeneous type for each $i$.
\item If $k=4$ then $b_2^+(X)\equiv 4$ mod $8$.
\end{enumerate}
If $X$ admits another gluing decomposition $X^\prime=X^\prime_1\cup_{Y^\prime_1}\cdots\cup_{Y^\prime_{l-1}} X^\prime_{l}$ along rational homology $3$-spheres of SWF-spherical type such that $b_2^+(X_j^\prime)\geq 1$ and $X_j^\prime$ is of $0$-dimensional BF homogeneous type for each $j$, then $k=l$ and $X^\prime$ satisfies the conditions (1) and (4).
Moreover,  there is a \Spinc structure on $X^\prime$ satisfying  (2).
\end{theo}

The following corollary is  a special case of \cref{thm:0decomp}.

\begin{cor}\label{cor:conndecomp}
Suppose closed $4$-manifolds $X_1,\ldots, X_k$ and their connected sum $X=X_1\#\cdots\#X_k$ ($1\leq k \leq 4$) satisfy the following conditions:
\begin{enumerate}
\item $b_1(X_i)=0$, $b_2^+(X_i)\equiv 3$ mod $4$ for each $i$.
\item There is a \Spinc structure $\fraks_i$ such that $d(\fraks_i)=0$ and $\SW(X_i,\fraks_i)$ is odd for each $i$.
\item $X_i$ is of $0$-dimensional BF homogeneous type for each $i$.
\item If $k=4$, then $b_2^+(X)\equiv 4$ mod $8$.
\end{enumerate}
If $X$  is diffeomorphic to another connected sum $X^\prime=X^\prime_1\#\cdots\# X^\prime_{l}$  such that $b_2^+(X_j^\prime)\geq 1$ and $X_j^\prime$ is of $0$-dimensional BF homogeneous type for each $j$, then $k=l$ and $X^\prime$ satisfies the conditions (1), (2) and (4).
\end{cor}
\begin{rem}\label{rem:symp}
A closed symplectic $4$-manifold with $b_2^+\geq 2$ satisfies the conditions (2) and (3) by Taubes' non-vanishing theorem \cite{Taubes0} and  \cref{prop:symBFhom}.
\end{rem}

\begin{cor}\label{cor:symp-sum}
Suppose closed symplectic $4$-manifolds $X_1,\ldots, X_k$ and their connected sum $X=X_1\#\cdots\#X_k$ ($1\leq k \leq 4$) satisfy the following conditions:
\begin{enumerate}
\item $b_1(X_i)=0$, $b_2^+(X_i)\equiv 3$ mod $4$ for each $i$.
\item If $k=4$, then $b_2^+(X)\equiv 4$ mod $8$.
\end{enumerate}
If $X$  is diffeomorphic to another connected sum $X^\prime=X^\prime_1\#\cdots\# X^\prime_{l}$ of symplectic manifolds $X^\prime_1,\ldots X^\prime_{l}$, then $k=l$ and  the conditions (1), (2) are satisfied.
\end{cor}

We propose a conjecture which extends the simple type conjecture.
\begin{conj}
Suppose that $X$ is a closed smooth oriented $4$-manifold with  
$b_2^+\geq2$.
Then, 
\begin{enumerate}
\item $X$ is of  BF blowup simple type. 
\item $X$ is of  BF homogeneous type.
\end{enumerate}
\end{conj}

The organization of the paper is as follows.
In Section~\ref{sec:blowup}, we review the Bauer--Furuta invariant and prove \cref{thm:negdef,thm:vanishing,thm:blowup,thm:rbd,thm:immersedadjunction}.
In Section~\ref{sec:BFdim}, we introduce the notions of BF dimension and BF homogeneous type, and prove \cref{thm:decomp,thm:0decomp}.
In Section~\ref{Y:sec:blow-up}, we introduce the notion of BF blowup simple type, and prove \cref{Y:claim:simple:immersed,thm:log}.
In Section~\ref{cohomotopy}, we discuss the stable cohomotopy groups of complex projective spaces, and prove some results used in the proof of \cref{thm:vanishing}.
Furthermore, we give an alternative proof of \cite[Lemma 3.5]{BF} which determines the Hurewicz maps.

\subsection*{Acknowledgements}

The authors were supported in part by JSPS KAKENHI Grant Numbers 23K22394 (Kato), 22K03284 (Kishimoto),   24K06716 (Nakamura)  and 23K03090 (Yasui).

\section{A generalization of the blowup formula and its applications}\label{sec:blowup}
 \subsection{Review of the Bauer--Furuta invariants}\label{subsec:BF}
 Let $Y$ be a closed $3$-manifold with \Spinc structure $\fraks_Y$.
If $b_1(Y)=0$, then  the Seiberg--Witten Floer homotopy type $\SWF(Y,\fraks_Y)$   is defined as an isomorphism class of an object in a certain $\UU(1)$-equivariant Spanier--Whitehead category $\frak{C}$ \cite{ManolescuSWF, Manolescu-gluing}.

\begin{defi}[\textit{Cf.}{\cite[Definition 3.34]{KMT}}]\label{def:SWF-spherical}
We say that a closed $3$-manifold $Y$ is \textit{SWF-spherical} for a \Spinc structure $\fraks_Y$ if $\SWF(Y,\fraks_Y)$ is represented by a sphere $S^{n\C}$ for some $n$.
\end{defi}

For a SWF-spherical $(Y,\fraks_Y)$, we can represent $\SWF(Y,\fraks_Y)$  by 
\[
\SWF(Y,\fraks_Y)=S^{-n(Y,\fraks_Y,g_Y)\C},
\]
where 
\[
n(Y,\fraks_Y,g_Y) = \ind_{\C} D(X,\fraks)-\frac18(c_1(\fraks)^2-\sign(X)),
\]
$X$ is a compact $4$-manifold with $\partial X=Y$, $D(X,\fraks)$ is a Dirac operator on a \Spinc strcuture on $X$ with $\fraks|_{\partial X}=\fraks_Y$, and $\ind_{\C} D(X,\fraks)$ is the index of $D(X,\fraks)$ with spectral boundary condition in the sense of Atiyah-Patodi-Singer \cite{APS1}.

\begin{defi}\label{def:spherical}
Let $X$ be a compact $4$-manifold with  possibly disconnected boundary $Y=\partial X$. 
We say that  the boundary $Y$ is of \textit{SWF-spherical type} if $b_1(Y)=0$ and each component of $Y$ is SWF-spherical for every \Spinc structure on the component.
\end{defi}

Let $X$  be a compact oriented $4$-manifold with boundary of SWF-spherical type.
Let $Y=\partial X=Y_1\cup\cdots \cup Y_k$.
Let $\fraks$ be a \Spinc structure on $X$ and $\fraks_Y=\fraks|_Y$.
Then $\SWF(Y,\fraks_Y)=\bigwedge_{i=1}^k\SWF(Y_i,\fraks|_{Y_i})$.
The Dirac operators associated with $\fraks$ difine a virtual index  bundle  $\Ind D$ over the  torus $Pic(X)=H^1(X;\R)/H^1(X;\Z)$.
Let $\lambda$ be the class of the $KO$-group given by 
\[
\lambda = \Ind D -\left[\underline{H^+(X;\R)}\right]\in KO_{S^1}(Pic(X)),
\]
where $\underline{H^+(X;\R)}$ is the trivial bundle over $Pic(X)$ with fiber $H^+(X;\R)$.
The relative Bauer--Furuta  invariant $\BF(X,\fraks)$ of $(X,\fraks)$ is defined as a  $\UU(1)$-equivariant stable morphism between objects in the cagegory $\frak{C}$, and it is represented by a $\UU(1)$-map,
\[
\Psi_{(X,\fraks)}\colon TF \to S^V\wedge \Sigma^*\SWF(Y,\fraks_Y),
\]
where $TF$ is the Thom space of a vector bundle over $Pic(X)$ and $V$ is a vector space such that 
$
\lambda=F-\underline{V}.
$
An important property of $\Psi_{(X,\fraks)}$ is that the restriction of $\Psi_{(X,\fraks)}$ to the $\UU(1)$-fixed point sets is induced from a fiberwise linear inclusion.
(See \cite{ManolescuSWF, Manolescu-gluing}.)
Since $\SWF(Y,\fraks_Y)$ is a sphere in this case, $\BF(X,\fraks)$ can be considered as a class in a  $\UU(1)$-equivariant stable cohomotopy group $\pi^0_{\UU(1)}(Pic(X);\lambda)$.
Therefore many results of Bauer and Furuta \cite{BF,BF2,B_survey} on closed $4$-manifolds are straightforwardly generalized to the case of the relative Bauer--Furuta invariants of $(X,\fraks)$ as above.

\begin{theo}[{\cite[Theorem 1.1]{BF}}]\label{thm:BF-SW}
If $b_2^+(X)>b_1(X)+1$, then a homology orientation determines a homomorphism $\pi^0_{\UU(1)}(Pic(X);\lambda)\to \Z$ which maps $\BF(X,\fraks)=[\Psi_{(X,\fraks)}]$ to the integer valued relative Seiberg--Witten invariant $\SW(X,\fraks)$.
\end{theo}
\begin{theo}[{\cite[Proposition 3.4]{BF}}]
If $b_1(X)=0$ and $b_2^+(X)\geq 2$, then $\BF(X,\fraks)$ can be considered as an element of the (non-equivariant) stable cohomotopy group $\pi^{b_2^+(X)-1}(\CP^{l-1})$, where $l=\ind_\C D(X,\fraks)$.
\end{theo}
\begin{rem}
The cohomotopy groups $\pi^{i}(\CP^{l-1})$ are computed in \cite[Lemma 3.5]{BF}. 
An elementary proof of \cite[Lemma 3.5]{BF} is given in \cref{cohomotopy}.
\end{rem}
Let 
\[
d(\fraks)=2\ind_\C D(X,\fraks)-(1-b_1(X)+b_2^+(X)).
\]
When $X$ is closed, this  is equal to $d(\fraks)$ in \eqref{eq:SWdim}. 
Let $X=X_1\cup_{Y_1}X_2\cup_{Y_2}\cdots\cup_{Y_{m-1}} X_m$ be a gluing decomposition along rational homology $3$-spheres of SWF-spherical type.
Let $\fraks$ be a \Spinc structure on $X$ and $\fraks_i=\fraks|_{X_i}$.
Then, by the gluing formula due to Manolescu \cite{Manolescu-gluing} (see also \cite{Tirasan, Sasahira, KLS}), the Bauer--Furuta invariant of $(X,\fraks)$ is given by
\begin{equation}\label{eq:BFgluing}
\BF(X,\fraks)=[\Psi_{(X_1,\fraks_1)}\wedge\cdots\wedge\Psi_{(X_m,\fraks_m)}].
\end{equation}
Bauer's result \cite[Proposition 4.5]{BF2} for connected sums is straightforwardly extended to gluing decompositions along rational homology $3$-spheres of SWF-spherical type.
\begin{theo}[{\cite[Proposition 4.5]{BF2}}]\label{thm:BF-nonvan}
For $(X,\fraks)$ as above,
suppose $m\geq 2$, $b_1(X_i)=0$  and $d(\fraks_i)=0$ for $i=1,\ldots,m$.
Then $\BF(X,\fraks)\neq 0$ if and only if the following conditions are satisfied:
\begin{itemize}
\item $b_2^+(X_i)\equiv 3$ mod $4$ for each $i$.
\item $\SW(X_i,\fraks_i)$ is odd for each $i$.
\item $m\leq4$. Moreover,   $b_2^+(X)\equiv 4$ mod $8$, if $m=4$.
\end{itemize}
Furthermore $(X,\fraks)$ satisfies Condition ($\ast$).
\end{theo}
\begin{cor}
For $X=X_1\cup_{Y_1}\cdots\cup_{Y_{m-1}} X_m$ as above,
suppose  that each $X_i$ is of the Seiberg--Witten simple type and $b_1(X_i)=0$.
Then, for every \Spinc structure $\fraks$ with  $\BF(X,\fraks)\neq 0$, $(X,\fraks)$ satisfies Condition ($\ast$).
\end{cor}

Let $N$ be a compact oriented $4$-manifold with $b_1(N)=0$ and $b_+(N)=0$ whose boundary is of SWF-spherical type.
For a \Spinc structure on $N$, the relative Bauer--Furuta invariant on $(N,\fraks)$ 
is represented by a $\UU(1)$-equivariant map
\[
\Psi_{(N,\fraks)}\colon (\C^{m+l}\oplus\R^n)^+ \to (\C^m\oplus\R^{n})^+,
\]
where $l=\ind_{\C} D(N,\fraks)$ and $m$, $n$ are some integers.
As mentioned above, the restriction of $\Psi_{(N_i,\fraks)}$ to the $\UU(1)$-fixed point sets is induced from a linear inclusion, and therefore a degree-one map.
Then we may assume $\Psi_{(N_i,\fraks)}$ itself is induced from an inclusion by the following lemma.
\begin{lem}[{\cite[Proposition 4.1]{BF2}}]\label{lem:inclusion}
Let $f\colon (\C^{m+k}\oplus\R^n)^+ \to (\C^m\oplus\R^{n})^+$ be an $\UU(1)$-equivariant map such that the induced map on the fixed point sets has degree 1.
Then $k\leq0$ and $f$ is $\UU(1)$-homotopic to the inclusion.
\end{lem}

Furthermore, 
the Bauer--Furuta invariants of $4$-manifolds with $b_2^+=0$ are always nontrivial.
\begin{prop}\label{rem:b+0}
Let $N$ be a closed oriented $4$-manifold or  a compact oriented $4$-manifold  whose boundary is of SWF-spherical type with  $b_2^+(N)=0$ and possibly positive $b_1(N)$.
Then  $\BF(N,\fraks)$ is nontrivial for arbitrary \Spinc structure $\fraks$.
\end{prop}
\begin{proof}
Note that  $\BF(N,\fraks)$ is induced from a $\UU(1)$-map $\bar\Psi\colon F\to V$ satisfying the following properties.
\begin{enumerate}
\item $V\cong\R^n$ and $F=F_0\oplus\underline{V}$, where $F_0$ is a complex vector bundle over $Pic(N)$ and $\underline{V}$ is a trivial bundle over $Pic(N)$ with fiber $V$.
\item The restriction  of $\bar\Psi$ to the fixed point sets  is a fiberwise linear isomorphism. 
\end{enumerate}
Then $\BF(N,\fraks)$ is the class of the extension $\Psi\colon TF\to S^V$ of $\bar\Psi$.
If $\Psi$ is $\UU(1)$-homotopic to $0$, then the restriction of $\Psi$ to the fixed point sets is homotopic to $0$.
But this is impossible by (2).
\end{proof}

On the other hand, a $4$-manifold with  $b_2^+=1$ and  $b_1=0$ has trivial Bauer--Furuta invariant.  
\begin{prop}[{\cite[\S4.2]{BF2}}]\label{prop:b+1}
Let $X$ be a closed oriented $4$-manifold or  a compact oriented $4$-manifold  whose boundary is of SWF-spherical type with  $b_2^+(X)=1$ and  $b_1(X)=0$.
Then  $\BF(X,\fraks)$ is trivial for arbitrary \Spinc structure $\fraks$.
\end{prop}
\begin{proof}
The Bauer--Furuta invariant of $(X,\fraks)$ is represented by a $\UU(1)$-equivariantly null-homotopic map $(\C^l)^+\to(\R)^+$.
\end{proof}

\subsection{Proof of \cref{thm:negdef}}

In this subsection, we prove \cref{thm:negdef} by giving a relation between $\BF(X_1,\fraks_1)$ and $\BF(X_2,\fraks_2)$.

Let $\iota\colon S^0\to (\C^\delta)^+$ be the $\UU(1)$-equivariant inclusion for a nonnegative integer $\delta$.
Define the homomorphism  $I_{\delta}\colon \pi^0_{S^1}(Pic(X);\lambda)\to\pi^0_{S^1}(Pic(X);\lambda-[\underline{\C}^{\delta}]) $ by $[\Phi]\mapsto [\Phi\wedge\iota]$.
Note that $I_0$ is the identity map.
Now the relation between $\BF(X_1,\fraks_1)$ and $\BF(X_2,\fraks_2)$ is given as follows.
\begin{theo}\label{thm:gblowup}
If $d(\fraks_1)\leq d(\fraks_2)$, then $\BF(X_1,\fraks_1)=I_{\delta} \BF(X_2,\fraks_2)$, where $\delta=d(\fraks_2)-d(\fraks_1)$.
\end{theo}

\begin{proof}
For $i=1,2$, let $k_i=\ind D(X_i,\fraks_i)$ and $l_i=\ind D(N_i,\fraks_i|_N)$.
Suppose $d(\fraks_1)\leq d(\fraks_2)$.
Then $l_1\leq l_2\leq 0$ and $\delta=k_2-k_1 = l_2-l_1 =|l_1|-|l_2|$.
Let $\Phi_i=\Psi_{(X,\fraks)}\wedge\Psi_{(N_i,\fraks_i|_{N_i})}$.
Let $\iota\colon S^0 \to(\C^\delta)^+$ be the inclusion.
Then we may assume $\bar\Psi_{(N_1,\fraks_i|_{N_1})} = \bar\Psi_{(N_2,\fraks_i|_{N_2})}\wedge\iota$ by \cref{lem:inclusion}, and therefore
\[
\BF(X_1,\fraks_1)=[\Phi_1] = [\Phi_2\wedge \iota]=I_\delta\BF(X_2,\fraks_2).
\]
\end{proof}

\begin{proof}[Proof of \cref{thm:negdef}]
If $d(\fraks_1)\leq d(\fraks_2)$ and $\BF(X_2,\fraks_2)=0$, then $\BF(X_1,\fraks_1)=I_\delta\BF(X_2,\fraks_2)=0$.
If $d(\fraks_1) = d(\fraks_2)$, then $\delta=0$ and $I_0$ is the identity map, and therefore  $\BF(X_1,\fraks_1)=\BF(X_2,\fraks_2)$
\end{proof}
\begin{rem}\label{rem:I}
If $b_1(X_i)=0$, then $Pic(X_i)=*$ and $\lambda=k_i[\C]-b[\R]$ where $b=b_+^2(X_i)$.
Moreover if $k_i>0$, then, $\pi^0_{S^1}(*, \lambda)=\pi^{b}_{S^1}(\ast,\C^{k_i})$ is isomorphic to the stable cohomotopy group $\pi^{b-1}(\CP^{k_i-1})$ by \cite[Proposition 3.4]{BF}.  
The map $I_{\delta}\colon \pi^0_{S^1}(\ast;\lambda)\to\pi^0_{S^1}(\ast;\lambda-[\underline{\C}^{\delta}])$ is identified with the map  $\pi^{b-1}(\C P^{k_2-1})\to\pi^{b-1}(\C P^{k_2-\delta-1})$ induced from the inclusion $\C P^{k_2-\delta-1}\to\C P^{k_2-1}$, which will be studied in \cref{naturality}.
\end{rem}

\subsection{Proof of \cref{thm:vanishing}}
For $i=1,2$, let $X_i$ be a compact oriented $4$-manifold with boundary of SWF-spherical type.
(We assume $b_2^+(X_i)$  for $i=1,2$ is arbitrary. Until \cref{cor:basic} below, $b_1(X_i )$ may be nonzero.)
Suppose we have a gluing decomposition  $X = X_1\cup_{Y}X_2$ along $Y$ of SWF-spherical type.
Let $\fraks$ be a \Spinc structure on $X$ and $\fraks_i:=\fraks|_{X_i}$.
The gluing formula \eqref{eq:BFgluing} immediately implies the following proposition and its corollaries.
\begin{prop}\label{prop:nonvanishing}
If $\BF(X,\fraks)\neq 0$, then $\BF(X_1,\fraks_1)\neq 0$ and $\BF(X_2,\fraks_2)\neq 0$.
\end{prop}
Let $\calB^{\Q}(X)$ be the set of Bauer--Furuta \textit{rational} basic classes, that is, the set of Bauer--Furuta  basic classes considered in the rational cohomology group $H^2(X;\Q)$.
\begin{cor}\label{cor:rationalbasic}
The inclusion $\calB^{\Q}(X_1\cup_Y X_2)\subset \calB^{\Q}(X_1)\oplus \calB^{\Q}(X_2)$ holds.
Suppose further that $Y$ is an integral homology $3$-sphere. 
Then $\calB(X_1\cup_Y X_2)\subset \calB(X_1)\oplus \calB(X_2)$.
\end{cor}

The following two corollaries are special cases of \cref{prop:nonvanishing} and \cref{cor:rationalbasic}. 
\begin{cor}[{\cite[Theorem 8.5]{B_survey}}]\label{cor:nonvanishing}
Let $X_1$, $X_2$ be  closed oriented smooth $4$-manifolds.
Let $\fraks_i$ be a \Spinc structure on $X_i$ for each $i$.
If $\BF(X_1\#X_2,\fraks_1\#\fraks_2)\neq 0$, then $\BF(X_1,\fraks_1)\neq 0$ and $\BF(X_2,\fraks_2)\neq 0$.
\end{cor}
\begin{cor}\label{cor:basic}
The inclusion $\calB(X_1\#X_2)\subset \calB(X_1)\oplus \calB(X_2)$ holds.
\end{cor}

Let $X=X_0\cup_Y N$ be a gluing decomposition along $Y$ of SWF-spherical type.
Suppose  $b_1(X_0)=0$, $b_2^+(X_0)\geq2$,  $b_1(N)=0$ and $b_2^+(N)=0$.
Let  $\fraks$ be a \Spinc structure on $X$.
\begin{prop}\label{prop:1-dim-vanishing}
If $d(\fraks)=1$ and  $\ind D(N,\fraks|_N)<0$, then $\BF(X,\fraks)=0$.
\end{prop}
\begin{proof}
Suppose $\BF(X,\fraks)\neq0$.
Then $b_2^+(X)\equiv 2$ mod $4$ and $\BF(X,\fraks)$ is a generator of the cohomotopy group  $\pi^{2k-3}(\CP^{k-1})$ for $k=\frac12(b_2^+(X)+2)$ which is represented by the square of the Hopf map (\cite[Lemma 3.5]{BF}, see also \cref{cohomotopy}).
Let $K$ be a $K3$ surface and $\fraks_0$ the canonical \Spinc structure on $K$.
By \cref{thm:BF-nonvan} or  \cite[Proposition 4.5]{BF2}, we obtain $\BF(X\#K,\fraks\#\fraks_0)\neq 0$.
If we assume $X\#K = X_0\cup_Y (N\#K)$, then $\BF(N\#K,\fraks|_N\#\fraks_0)\neq0$ by 
\cref{prop:nonvanishing}.  
On the other hand, the assumption $\ind D(N,\fraks|_N)<0$ implies that $d(\fraks|_N\#\fraks_0)<0$ and therefore $\BF(N\#K,\fraks|_N\#\fraks_0) = 0$. 
This is a contradiction.
\end{proof}
\begin{proof}[Proof of \cref{thm:vanishing}]
If $d(\fraks_2)\leq 1$, then $d(\fraks_1)\leq -1$. Therefore $\BF(X_1,\fraks_1)=0$.
If $d(\fraks_2)=2$, then $d(\fraks_1) =0$ or $d(\fraks_1)\leq -2$.
In the latter case, we have $\BF(X_1,\fraks_1)=0$.
If $d(\fraks_2)=2$ and $d(\fraks_1) =0$, then $k_2-k_1=1$ and $\BF(X_1,\fraks_1)=I_2\BF(X_2,\fraks_2)$, where $k_i=\ind_{\C}D(X_i,\fraks_i)$.
As noted in \cref{rem:I}, $I_2$ is identified with the map
\[
I_2\colon \pi^{2k_2-4}(\CP^{k_2-1})\to \pi^{2k_2-4}(\CP^{k_2-2}).
\]
Since $\pi^{2k_2-4}(\CP^{k_2-1})=\Z\oplus \Z/m$ $(m=\gcd(2,k_2))$ and $\pi^{2k_2-4}(\CP^{k_2-2})=\Z$ by \cite[Lemma 3.5]{BF} (see also \cref{cohomotopy}), 
and $\BF(X_2,\fraks_2)$ is a torsion by the assumption,  we have $\BF(X_1,\fraks_1)=I_2\BF(X_2,\fraks_2)=0$.
If $d(\fraks_2)=3$, then $d(\fraks_1) =1$ or $d(\fraks_1)\leq -1$.
In the latter case, we have $\BF(X_1,\fraks_1)=0$.
If $d(\fraks_2)=3$ and $d(\fraks_1) =1$, then the map $I_2\colon\pi^{2k_2-5}(\CP^{k_2-1})\to \pi^{2k_2-5}(\CP^{k_2-2})$ is trivial by Proposition \ref{naturality k=5,6}, and therefore $\BF(X_1,\fraks_1)=I_2\BF(X_2,\fraks_2)=0$.
The vanishing of $\BF(X_1,\fraks_1)$ is also obtained by \cref{prop:1-dim-vanishing} since $l_1<l_2\leq 0$, where $l_i=\ind_{\C} D(N_i,\fraks_i|_{N_i})$.
\end{proof}

\begin{proof}[Proof of \cref{thm:blowup} and \cref{thm:rbd}]
These are direct consequences of \cref{thm:negdef} and \cref{thm:vanishing}. 
\end{proof}

Let $X$ be a closed oriented $4$-manifold with $b_1(X)=0$, $b_2^+(X)\geq 2$.
Let $\fraks$ be a \Spinc structure on $X$.
Since $k\overline{\CP}^2$ is  simply-connected, the \Spinc structures on $k\overline{\CP}^2$  are characterized by their first Chern classes.
For $\mathbf{c}:=(c_1,\ldots,c_k)$ such that every $c_i$ is odd,  let $\fraks_{\mathbf{c}}$ be the \Spinc structure  on $k\overline{\CP}^2$ such that $c_1(\fraks_{\mathbf{c}})=\sum_{i=1}^kc_iE_i$, where $E_i$ are exceptional classes.
\begin{prop}\label{prop:blowupbasic}
Suppose $(X,\fraks)$ satisfies Condition $(\ast)$ and $\BF(X,\fraks)\neq 0$. 
Then $\BF(X\#k\overline{\CP}^2,\fraks\#\fraks_{\mathbf{c}})\neq 0$ if and only if $c_i=\pm1$ for every $i$.
\end{prop}
\begin{proof}
This is proved inductively by \cref{thm:blowup}.
\end{proof}

For later purpose, we recall Kronheimer's adjunction inequality \cite{Kr99}.
For a smoothly embedded surface $\Sigma$ in a smooth $4$-manifold $X$, let 
\[
\chi_-(\Sigma)=\sum_{g_i>0}(2g_i-2),
\]
where $g_i$ is the genus of the component $\Sigma_i$ of $\Sigma$.
\begin{defi}
A class $K\in H^2(X;\Z)$ is a \textit{monopole class}, if $K$ is the first Chern class $c_1(\fraks)$ of some \Spinc structure $\fraks$ on $X$ with the property that the Seiberg--Witten equations on $\fraks$ have a solution for all Riemannian metric on $X$.
\end{defi}
\begin{theo}[Kronheimer \cite{Kr99}]\label{thm:kron}
Let $X$ be a closed oriented smooth $4$-manifold.
Suppose that  $\Sigma$ is an embedded surface representing a class $\alpha\in H_2(X;\Z)$ with non-negative square $m=\alpha\cdot\alpha$.
For $X_m=X\#m\overline{\CP}^2$, let $K_m=K+E_1+\cdots +E_m$.
If $K_m$ is a monopole class on $X_m$,
 then
\begin{equation}\label{eq:kron}
\chi_-(\Sigma)\geq \alpha\cdot\alpha + |\alpha\cdot K|.
\end{equation}
\end{theo}

\begin{cor}\label{cor:adj}
Let $X$ and $\Sigma$ be as in \cref{thm:kron}.
If $K$ is a Bauer--Furuta basic class, then the inequality \eqref{eq:kron} holds.
\end{cor}
\begin{proof}
The proof follows from \cref{thm:kron} and \cref{prop:basic} below.
\end{proof}

\begin{prop}\label{prop:basic}
If $K$ is a Bauer--Furuta basic class, then $K_m$ is a monopole class for each non-negative $m$.
\end{prop}
\begin{proof}
If the Seiberg--Witten equations have no solution for some Riemannian metric, then $\Psi^{-1}(0)$ is empty, where
$\Psi$ is the $\UU(1)$-map  representing the Bauer--Furuta invariant.
This implies that $\Psi$ is $\UU(1)$-homotopic to the trivial map, and therefore the Bauer--Furuta invariant is trivial.
Hence, a Bauer--Furuta basic class is a monopole class.
By \cref{thm:blowup}, $K_m$  is a Bauer--Furuta basic class for each $m$, and therefore it is a monopole class.
\end{proof}

\subsection{Proof of \cref{thm:immersedadjunction}}

\begin{prop}\label{prop:nonneg}
Suppose $X$ is a smooth closed $4$-manifold with an embedded sphere $S$ with self-intersection $-r <0$.
Let $\fraks$ be a \Spinc structure with $\BF(X,\fraks)\neq 0$.
If $|\langle c_1(\fraks),[S]\rangle|+[S]\cdot[S]\geq 0$,
then 
\begin{align*}
\BF(X,\fraks+\alpha^*)&\neq 0 \quad \text{ if } \langle c_1(\fraks),[S]\rangle>0\\
\BF(X,\fraks-\alpha^*)&\neq 0 \quad \text{ if } \langle c_1(\fraks),[S]\rangle<0,
\end{align*}
where $\alpha^*$ is the Poincar\'{e} dual of $[S]$.
\end{prop}
\begin{proof}
Suppose $\langle c_1(\fraks),[S]\rangle>0$.
Let $N$ be a tubular neighborhood of $S$. 
Note that the boundary of $N$ is a lens space which admits a positive scalar curvature metric.
Decompose $X$ as
\[
X=X_0\cup N,
\]
where $X_0=\overline{X\setminus N}$.
Note that $\fraks|_{X_0} = (\fraks+\alpha^*)|_{X_0}$, and
\[
d(\fraks+\alpha^*) = d(\fraks)+\langle c_1(\fraks),[S]\rangle+[S]\cdot[S]\geq d(\fraks).
\]
By \cref{thm:negdef}, we obtain $\BF(X,\fraks+\alpha^*)\neq 0$.
The proof of the case when $\langle c_1(\fraks),[S]\rangle<0$ is similar.
\end{proof}
\begin{proof}[Proof of \cref{thm:immersedadjunction}]
Suppose that $x$ is represented by an immersed sphere with $p$ positive and $n$ negative double points.
Suppose $\langle K,\alpha\rangle = \langle c_1(\fraks),\alpha\rangle\geq0$.
Then in $\overline{X}=X\#(p+n)\overline{\CP}^2$, $\bar{\alpha} = \alpha - 2\sum_{j=1}^p e_j$ is represented by an embedded sphere, where $e_j$ is the homology class of the exceptional divisor in the $j$-th $\overline{\CP}^2$.
Let $E_j$ be the Poincar\'{e} dual of $e_j$ and $\bar{\fraks}$ the \Spinc structure with $c_1(\bar{\fraks})=c_1(\fraks)+\sum_{j=1}^{p+n}E_j$.
By \cref{prop:nonneg}, either $\bar{\alpha}\cdot\bar{\alpha}+\langle c_1(\bar{\fraks}),\bar{\alpha}\rangle\leq -2$ or $\BF(\overline{X},\bar{\fraks})\neq 0$.
Now 
\[
\bar{\alpha}\cdot\bar{\alpha}=\alpha\cdot\alpha-4p,\quad \langle c_1(\bar{\fraks}),\bar{\alpha}\rangle =\langle c_1(\fraks),\alpha\rangle+2p.
\]
Therefore in the first case,
\[
\alpha\cdot\alpha+\langle c_1(\fraks),\alpha\rangle\leq 2p-2.
\]
Otherwise
\[
\BF(\overline{X},\bar{\fraks})\neq 0.
\]
Furthermore, it follows from \cref{thm:blowup} that
\[
\BF(\overline{X},\bar{\fraks}) = \BF(X,\fraks).
\]
The proof of the case  when $\langle K,\alpha\rangle \leq0$ is similar.
\end{proof}

\section{BF dimension and BF homogeneous type}\label{sec:BFdim}
We introduce a diffeomorphism invariant, BF dimension, which is rather weaker than the Bauer--Furuta invariant, but might be easier to handle.
\begin{defi}\label{def:BFdim}
Let $X$ be a closed  $4$-manifold or a compact  $4$-manifold with boundary of SWF-spherical type. 
(Here $b^+_2(X)$ is arbitrary). 
Define $\Delta(X)$ by
\[
\Delta(X)=\{d(\fraks)\,|\, \fraks\text{ is a \Spinc structure on $X$ with }\BF(X,\fraks)\neq 0\}.
\]
Then we define,
\[
\BFdim(X)=
\begin{cases}
\min\Delta(X), & \text{ if $\Delta(X)$ has a minimum element,}\\
-\infty, &\text{ if $\Delta(X)\neq \emptyset$ and $\Delta(X)$ has no lower bound,} \\
\infty, & \text{ if $\Delta(X)= \emptyset$.}
\end{cases}
\]
\end{defi}

\begin{prop}\label{prop:BFdim-nonneg}
Let $X$ be a  $4$-manifold  in \cref{def:BFdim}.
Then $\BFdim(X)=-\infty$ or $-1\leq \BFdim(X)\leq \infty$.
More precisely, the following holds.
\begin{enumerate}
\item If $b_2^+(X)=0$ and $b_2^-(X)\neq 0$, then $\BFdim(X)=-\infty$.
If $b_2^+(X)=0$ and $b_2^-(X)= 0$, then $\BFdim(X)=b_1(X)-1$.
\item If  $b^+_2(X)\geq 1$, then $0\leq \BFdim(X)\leq +\infty$.  
In particular, if $b^+_2(X)= 1$ and $b_1(X)=0$, then $\BFdim(X)=+\infty$.
\end{enumerate}
\end{prop}
\begin{proof}
(1) If $b_2^+(X)=0$, then $\BF(X,\fraks)\neq 0$ for every \Spinc structure $\fraks$ on $X$ (\cref{rem:b+0}).
If $b_2^-(X)\neq 0$,  then $\Delta(X)$ has no lower bound.
If $b_2^-(X) = 0$,  then $d(\fraks)=b_1(X)-1$ for every $\fraks$.

(2) Let $\fraks$ be a \Spinc structure on $X$.
If $b^+_2(X)\geq 1$ and $d(\fraks)<0$, then the monopole map for $(X,\fraks)$ is null-homotopic $\UU(1)$-equivariantly, and therefore $\BF(X,\fraks)=0$. 
This implies that  $\dim_{\BF}(X) \geq 0$.
If $b^+_2(X)= 1$ and $b_1(X)=0$, then \cref{prop:b+1} implies that $\Delta(X)=\emptyset$.
\end{proof}

\begin{prop}\label{prop:BFdim-sum}
Let $X=X_1\cup_Y X_2$ be a gluing decomposition along $Y$ of SWF-spherical type.
Suppose $\BFdim(X_i)=d_i$ for $i=1,2$.
Then 
\begin{enumerate}
\item If $d_1$ and $d_2$ are finite, then $d_1+d_2+1\leq \BFdim(X)\leq+\infty$.
\item If one of $d_i$ is $+\infty$, then $\BFdim(X)=+\infty$.
\end{enumerate}
\end{prop}
\begin{proof}
(1) For a \Spinc structure $\fraks$ on $X$, let $\fraks_i=\fraks|_{X_i}$ for $i=1,2$. 
Then $d(\fraks)=d(\fraks_1) +d(\fraks_2)+1$.
If $d(\fraks)\leq d_1+d_2$ for a \Spinc structure $\fraks$ on $X$, then $d(\fraks_i)<d_i$ for at least one of $i=1,2$, and therefore $\BF(X_i,\fraks_i)=0$.  
By \cref{prop:nonvanishing},  $\BF(X,\fraks) = 0$.

(2) If $d_1$ is $+\infty$, then $\BF(X_1,\fraks_1)=0$ for every \Spinc structure $\fraks_1$ on $X_1$. 
\cref{prop:nonvanishing} implies that $\BF(X,\fraks)=0$ for every \Spinc structure $\fraks$ on $X$ and therefore the BF dimension of $X$ is $+\infty$. 
\end{proof}

\begin{cor}\label{cor:BFdim-conn}
For each $i=1,2$, let $X_i$ be a closed $4$-manifold with  $\BFdim(X_i)=d_i$.
Then 
\begin{enumerate}
\item If $d_1$ and $d_2$ are finite, then $d_1+d_2+1\leq \BFdim(X_1\#X_2)\leq+\infty$.
\item If one of $d_i$ is $+\infty$, then $\BFdim(X_1\#X_2)=+\infty$.
\end{enumerate}
\end{cor}
\cref{prop:BFdim-sum} and \cref{prop:BFdim-nonneg} imply the following corollaries.
\begin{cor}\label{cor:BFdim-sum2}
Let $X=X_1\cup_{Y_1}\cdots\cup_{Y_{l-1}}X_l$ be a gluing decomposition along rational homology $3$-spheres of SWF-spherical type.
If $b_2^+(X_i)\geq 1$ for every $i$, then $l-1\leq\BFdim(X)\leq+\infty$.
\end{cor}

\begin{cor}\label{cor:BFdim-conn2}
Suppose closed oriented smooth $4$-manifolds $X_1,\ldots,X_l$ have positive  $b^+_2$.
Then  $l-1\leq\BFdim(\#_{i=1}^lX)\leq+\infty$.
\end{cor}
\begin{proof}[Proof of \cref{thm:decomp}]
By \cref{cor:BFdim-sum2}, we have $k-1\leq \BFdim(X)\leq+\infty$ and $l-1\leq\BFdim(X^\prime)\leq+\infty$. 
By \cref{thm:BF-nonvan}, $X$ has a \Spinc structure $\fraks$ such that $\BF(X,\fraks)\neq 0$ and $d(\fraks)=k-1$.
Therefore $\BFdim(X)=k-1$. 
If $l>k$, then $\BFdim(X^\prime)>\BFdim(X)$, and therefore $X$ is not diffeomorphic to $X^\prime$.

Since $b_1(X)=0$, we have $b_1(X_j^\prime)=0$ for every $j$.
Suppose $b_2^+(X_j^\prime)=1$ for some $j$. 
Then $\BF(X_j^\prime,\fraks|_{X_j^\prime})=0$ by \cref{prop:b+1}, and therefore $\BF(X,\fraks)\neq 0$. 
This is a contradiction.
Hence $b_2^+(X_j^\prime)\geq 2$ for every $j$.
\end{proof}

Next we introduce the notion of \textit{BF homogeneous type} which generalizes the Seiberg--Witten simple type.
\begin{defi}\label{def:BFhom}
Let $X$ be a closed  $4$-manifold or a compact  $4$-manifold with boundary of SWF-spherical type with $b_2^+(X)\geq 1$.
We call $X$ has the \textit{$d$-dimensional BF homogeneous type}, if  $\BF(X,\fraks)=0$ whenever $d(\fraks)\neq d$.
\end{defi}
%
\begin{rem}\label{rem:trivial}
By definition, the following conditions (1), (2) and (3) are equivalent:
\begin{enumerate}
\item $X$ is of the $d$-dimensional BF homogeneous type for more than two $d$'s.
\item $X$ is of the $d$-dimensional BF homogeneous type for every $d$.
\item $X$ has no \Spinc structure $\fraks$ with $\BF(X,\fraks)\neq0$.  
\end{enumerate}
\end{rem}
\begin{defi}
We call the $4$-manifold $X$ satisfying one of the equivalent conditions (1)(2)(3) in \cref{rem:trivial} is of the \textit{BF trivial type}.
\end{defi}
\begin{rem}
In the case when $b_2^+-b_1\geq 2$,
if $X$ is of the $0$-dimensional BF homogeneous type, then $X$ is of Seiberg--Witten simple type.
It would be interesting to ask whether the converse holds. 
\end{rem}
\begin{rem}
If $X$ is of the $d$-dimensional BF homogeneous non-trivial type  for some $\fraks$, then the BF-dimension of $X$ is $d$.
\end{rem}

Taubes proved that the symplectic $4$-manifolds are of  Seiberg--Witten simple type \cite[Theorem 0.2(6)]{Taubes}.
This is generalized to the following.
\begin{prop}\label{prop:symBFhom}
A closed symplectic $4$-manifold $X$ with $b_2^+\geq 2$ has the $0$-dimensional BF homogeneous type.
Furthermore, if $b_2^+-b_1\geq 2$, then $X$ is not of  the  BF trivial type.
\end{prop}
\begin{proof}
The proof of \cite[Theorem 0.2(6)]{Taubes} also works in this case.
There is a unique canonical \Spinc structure $\fraks_0$ on a symplectic $4$-manifold $X$, and every \Spinc structure on $X$ is written as $\fraks_0+E$ for some $E\in H^2(X;\Z)$.
If $\BF(X,\fraks_0+E)\neq 0$, then we can find an embedded symplectic curve which represents the Poincar\'{e} dual of $E$ by the proof of Theorem 0.1 of \cite{Taubes}.
Then the rest of the proof is same with that of \cite[Theorem 0.2(6)]{Taubes}.
Furthermore, if $b_2^+-b_1\geq 2$, then $\BF(X,\fraks_0)\neq 0$ by \cref{thm:BF-SW} and the non-vanishing theorem of the Seiberg--Witten invariant due to  Taubes \cite{Taubes0}.
\end{proof}

The next proposition is a consequence of \cref{cor:nonvanishing}.
\begin{prop}\label{prop:hom-conn}
Let $X=X_1\cup_Y X_2$ be a gluing decomposition along $Y$ of SWF-spherical type.
For each $i=1,2$, suppose $X_i$ is of the $d_i$-dimensional BF homogeneous type.
Then $X$ has the $(d_1+d_2+1)$-dimensional BF homogeneous type.
\end{prop}
\begin{proof}
Let $\fraks$ be a \Spinc structure on $X$ and $\fraks_i=\fraks|_{X_i}$  for each $i=1,2$.
By \cref{cor:nonvanishing}, if one of $\BF(X_i,\fraks_i)$ is zero, then $\BF(X,\fraks)$ is also zero.
Note that $d(\fraks) = d(\fraks_1)+d(\fraks_2)+1$.
Therefore the only Bauer--Furuta invariants of dimension $d_1+d_2+1$ can be non-zero.
\end{proof}

\begin{cor}\label{cor:sympconn}
Let $X_1,\ldots,X_k$ be closed symplectic $4$-manifolds with $b_2^+(X_i)\geq 2$.
Then $X_1\#\cdots\#X_k$ has the $(k-1)$-dimensional BF homogeneous type.
\end{cor}
\begin{proof}[Proof of \cref{thm:0decomp}]
By \cref{thm:BF-nonvan}, $\BF(X,\fraks)\neq 0$.
This fact with \cref{prop:hom-conn} implies that $X$ is of the $(k-1)$-dimensional BF homogeneous non-trivial type.
Since $X^\prime$ has the same properties, we have $k=l$, and $X^\prime$ satisfies the conditions in \cref{thm:BF-nonvan}.
\end{proof}

\section{Blowup simple type}\label{Y:sec:blow-up}
We introduce the notion of \textit{blowup simple type}, which generalizes Seiberg--Witten simple type from the viewpoint of the blowup formula. This property plays an important role for determining { the Bauer--Furuta invariants of 4-manifolds obtained by logarithmic transformations in a fishtail neighborhood.}

We first recall relevant definitions.

\begin{defi}\label{Y:def:SW simple type}
Suppose that $b_2^+(X)\geq 2$.
(1) A second cohomology class $K$ of $X$ is called  a {\it Seiberg--Witten basic class}, if there exists a \Spinc structure $\fraks$ on $X$ such that $c_1(\fraks)=K$ and $\mathrm{SW}_X(\fraks)\neq 0$, where $\mathrm{SW}_X(\fraks)$ denotes  the (integer valued) Seiberg--Witten invariant for $\fraks$.  Furthermore, if $\mathrm{SW}_X(\fraks)\not\equiv 0 \pmod{2}$, then $K$ is called a {\it mod $2$ Seiberg--Witten basic class}. 

(2) {The $4$-manifold $X$} is called of {\it Seiberg--Witten $($resp.\ mod $2$ Seiberg--Witten$)$ simple type}, if every Seiberg--Witten $($resp.\ mod $2$ Seiberg--Witten$)$ basic class $K$ satisfies $d_X(K)=0$, where $d_X(K)=\frac{1}{4}(K^2-2\chi(X)-3\sigma(X))$, and $\chi(X)$ and $\sigma(X)$ respectively denote the Euler characteristic and the signature of $X$. 
\end{defi} 

In the rest of this section, { for a positive integer $k$, let $e_1,\dots, e_k$ be the standard orthogonal basis of $H_2(k\overline{\CP}^2)=\oplus_k H_2(\overline{\CP}^2)$, and let each $E_i\in H^2(k\overline{\CP}^2)$ be the Poincar\'{e} dual of $e_i$. In the case when $k=1$, we often denote $e_1$ and $E_1$ respectively by $e$ and $E$.} Here we introduce the notion of blowup simple type. 

\begin{defi}\label{Y:def:blowup simple type}
(1) A Bauer--Furuta (resp.\ Seiberg--Witten) basic class $K$ of $X$ {is} called \textit{BF $($resp.\ SW$)$ simple}, if $K+lE$ is not a Bauer--Furuta (resp.\ Seiberg--Witten) basic class  of $X\#\overline{\CP}^2$ for any integer $l\neq \pm 1$. 
It {is} often called BF $($resp.\ SW$)$ simple in $X$, emphasizing that $K$ is a basic class of $X$. 

(2) The $4$-manifold $X$ {is} called of \textit{BF $($resp.\ SW$)$ blowup simple type}, if every Bauer--Furuta (resp.\ Seiberg--Witten) basic class is of BF $($resp.\ SW$)$ simple. 
\end{defi}

\begin{rem}
(1) If  $(X,\fraks)$ satisfies Condition ($\ast$) in \cref{sec:introduction} and $K=c_1(\fraks)$ is a Bauer--Furuta basic class, then $K$ is  BF simple by \cref{thm:blowup}. 
(2) If $b_2^+=0$, then every Bauer--Furuta basic class is \textit{not} BF simple. 
This can be seen from  \cref{rem:b+0}.
\end{rem}

The lemma below easily follows from the blowup formula for Seiberg--Witten invariants \cite[Theorem 1.4]{FSimmersed}. 

\begin{lem}\label{Y:claim:SW simple=BF}
Let $K$ be a Seiberg--Witten basic class of $X$. 
Consider the following conditions:
\begin{itemize}
 \item [(i)] $d_X(K)=0$.
 \item [(ii)] $K$ is SW simple. 
 \item [(iii)] $K$ is BF simple. 
\end{itemize}
Then the following {\rm(1)} and {\rm (2)} hold.
\begin{enumerate}
\item  If $b_2^+>1$, then the conditions {\rm (i)} and {\rm (ii)} are equivalent.
\item  If $b_2^+-b_1>1$, 
then $K$ is a Bauer--Furuta basic class, 
and {\rm (i), (ii)} and {\rm (iii)} are equivalent. 
\end{enumerate} 
\end{lem}
\begin{proof}By a result of Bauer and Furuta  \cite[Proposition 3.3]{BF}, if $b_2^+-b_1>1$, then every Seiberg--Witten basic class is Bauer--Furuta basic class, and hence $K$ is a Bauer--Furuta basic class. 
Also, a result of Fintushel and Stern \cite[Theorem 1.4]{FSimmersed} shows that the condition (i) is equivalent to (ii), if $b_2^+>1$. 
The claim (i) $\Rightarrow $ (iii) easily follows from the fact every Bauer--Furuta basic class $L$ of a 4-manifold $Z$ with positive $b_2^+$ satisfies $d_Z(L)\geq 0$. The claim (iii) $\Rightarrow $ (ii) follows from the fact that every Seiberg--Witten basic class is Bauer--Furuta basic class. 
\end{proof}

\begin{cor}\label{Y:claim:simple=blow-up}
A $4$-manifold $X$ is of Seiberg--Witten simple type if and only if $X$ is of SW blowup simple type. 
\end{cor}

Due to this corollary, BF blowup simple type gives another natural definition of simple type for Bauer--Furuta invariants. We will show that there are many examples of 4-manifolds having BF blowup simple type. Symplectic 4-manifolds are basic examples. 
\begin{prop}[cf.\ Bauer~\cite{B_survey}]\label{Y:claim:symplectic}
Every closed oriented symplectic $4$-manifold with {$b_2^+-b_1>1$} is of BF blowup simple type. 
\end{prop}
\begin{proof}
By \cref{prop:symBFhom}, a closed oriented symplectic $4$-manifold with $b_2^+>1$ has the $0$-dimensional BF homogeneous type, and therefore every Bauer--Furuta basic class $K$ satisfies that $d_X(K)=0$. 
If $b_2^+-b_1>1$, then $K$ is also a Seiberg--Witten basic class by \cref{thm:BF-SW}.
Then the claim  follows from \cref{Y:claim:SW simple=BF}. 
\end{proof}

We observe that blowup operation preserves BF and SW simplicities. 

\begin{lem}\label{Y:claim:blow-up}Suppose that a Bauer--Furuta $($resp.\ Seiberg--Witten$)$ basic class $K$ of $X$ is $BF$ $($resp.\ SW$)$ simple. Then both $K+E$ and $K-E$ are BF simple in $X\#\overline{\CP}^2$. 
\end{lem}

\begin{proof}The SW case can be easily proved using the fact that $d_X(K)=0$, so we prove the BF case.  Let $E'$ denotes the generator of the second cohomology group of the last connected summand $\overline{\CP}^2$ of $X\#\overline{\CP}^2\#\overline{\CP}^2$. Suppose that $K+ E+lE'$ is a Bauer--Furuta basic class of $X\#\overline{\CP}^2\#\overline{\CP}^2$. 
As easily seen, there exists an orientation preserving self-diffeomorphism of $X\#\overline{\CP}^2\#\overline{\CP}^2$ whose induced isomorphism maps $K+ E+lE'$ to $K+lE+E'$. Therefore, $K+lE+E'$ is a Bauer--Furuta basic class of $X\#\overline{\CP}^2\#\overline{\CP}^2$. Theorem~\ref{thm:blowup} thus shows that $K+ lE$ is a Bauer--Furuta basic class of $X\#\overline{\CP}^2$. Since $K$ is BF simple, we see $l=\pm 1$, showing that $K+E$ is BF simple. The same argument shows the claim for $K-E$.  
\end{proof}

In the rest of this section, let $X_1$ and $X_2$ be closed connected oriented smooth 4-manifolds. 
We observe the lemma below. 
\begin{lem}\label{Y:claim:non BF simple}
Let $K_1$ and $K_2$ be second cohomology classes of $X_1$ and $X_2$, respectively. If $K_1+K_2+lE$ is a Bauer--Furuta basic class of $X_1\#X_2\#\overline{\CP}^2$ for some integer $l\neq \pm 1$, then each $K_i$ is a Bauer--Furuta basic class of $X_i$ which is not BF simple. 
\end{lem}
\begin{proof}Since $X_1\#X_2\#\overline{\CP}^2$ is diffeomorphic to $(X_1\#\overline{\CP}^2)\#X_2$,  
	 \cref{cor:nonvanishing} shows that $K_1+lE$ and $K_1$ are Bauer--Furuta basic classes of respectively $X_1\#\overline{\CP}^2$ and $X_1$. Hence the assumption $l\neq \pm 1$ shows that $K_1$ is not BF simple in $X_1$. Also, the same argument shows the claim for $K_2$. 
\end{proof}

The proposition below is straightforward from this lemma, and this proposition and Proposition \ref{Y:claim:symplectic} immediately shows the corollary below.  
\begin{prop}\label{Y:claim:connected sum}
If at least one of $X_1$ and $X_2$ is of BF blowup simple type, then the connected sum $X_1\#X_2$ is of BF blowup simple type. 
\end{prop}

\begin{cor}\label{Y:claim:sum symplectic} If  $X_1$  admits a symplectic structure and satisfies that $b_2^+(X_1)-b_1(X_1)>1$, then the connected sum $X_1\#X_2$ is of BF blowup simple type. 
\end{cor}

This corollary and the connected sum formulae of Bauer~\cite{BF2} and Ishida and Sasahira \cite{IS} for Bauer--Furuta invariant produce many examples of smooth 4-manifolds of BF blowup simple type having the vanishing Seiberg--Witten invariant and yet a non-vanishing Bauer--Furuta invariant.

Now we prove \cref{Y:claim:simple:embedded,Y:claim:simple:immersed}.

\begin{proof}[Proof of \cref{Y:claim:simple:embedded}]
Let $\alpha$ be the second homology class of $X$ represented by the given surface.
Since $\alpha\cdot\alpha=2g-2>0$, $\alpha$ is  a non-torsion class and $b_2^+(X)\geq 1$.
Let $K$ be a Bauer--Furuta basic class of $X$. Suppose that $K+lE$ is a Bauer--Furuta basic class of $X\#\overline{\CP}^2$ for some integer $l{\neq}\pm 1$. Since $K+lE$ is characteristic, we see that $l$ is odd, showing $|l|>1$. 
{We may assume $\langle K,\, \alpha \rangle \geq 0$ if necessary by reversing the orientation of the surface. Since $\alpha-e$ is represented by a smoothly embedded surface of genus $g$ in $X\#\overline{\CP}^2$ with self-intersection number $2g-3>0$,  \cref{cor:adj} implies that
\begin{equation*}
\langle K+lE,\, \alpha-e \rangle  \,+\, (\alpha-e)\cdot(\alpha-e) \, \leq \, 2g-2.
\end{equation*}
The left hand side is equal to $\langle K,\, \alpha \rangle+l+2g-3$. In the case where $l>1$, this value is larger than $2g-2$, giving a contradiction to the above inequality. In the case where $l<-1$, by applying the adjunction inequality to $\alpha+e$, we obtain a contradiction. Therefore, the claim follows.} 
\end{proof}

\begin{proof}[Proof of \cref{Y:claim:simple:immersed}]
In the case where $p>1$, the claim follows from \cref{Y:claim:simple:embedded}, { since $X$ admits a smoothly embedded closed oriented surface of genus $p$ with self-intersection number $2p-2$.  
 We thus assume $p=1$. Let $K$ be a Bauer--Furuta basic class of $X$. We will show that $K$ is BF simple. To the contrary, suppose that $K+lE_1$ is a Bauer--Furuta basic class of $X\#\overline{\CP}^2$ for some integer $l$ with $l\neq \pm 1$. Then, we easily see $|l|>1$. Let $\alpha$ be a second homology class of $X$ represented by the immersed 2-sphere. If necessary by resolving the negative double points using blowing ups, we may assume that $\alpha$ is represented by an immersed 2-sphere with a single positive double point in $X\#k\overline{\CP}^2$ for some positive integer $k$.  
We may furthermore assume $\langle K,\alpha \rangle \geq 0$. Due to Theorem~\ref{thm:negdef}, the class $K+lE_1+E_2+\dots+E_k$ is a Bauer--Furuta basic class of $X\#k\overline{\CP}^2$.   
By \cref{Y:immersed:na} below, for any positive integer $n$, the second homology class $n\alpha-e_1$ of $X\#k\overline{\CP}^2$ is represented by an immersed 2-sphere with one positive double point. 
In the case where $l>1$, the inequality  
\begin{equation*}
\langle K+lE_1+E_2+\dots+E_k,\, n\alpha-e_1 \rangle  \,+\, (n\alpha-e_1)\cdot(n\alpha-e_1) \, > \, 2\cdot 1-2
\end{equation*}
holds, since the left hand side is equal to $n\langle  K, \alpha \rangle+ l-1$. 
\cref{thm:immersedadjunction} thus shows that $K+lE_1+E_2+\dots+E_k+2PD(n\alpha-e_1)$ is Bauer--Furuta basic class of $X\#k\overline{\CP}^2$ for any positive integer $n$, where $PD$ denotes the Poincar\'{e} dual. This shows that $X\#k\overline{\CP}^2$ has infinitely many Bauer--Furuta basic classes whose virtual dimensions are larger than that of $K+lE_1+E_2+\dots+E_k$, giving a contradiction to the finiteness of such classes (\cite[Lemma~2.4]{Ko04}). In the case where $l<-1$,  by applying the same argument to $n\alpha+e_1$, we obtain a contradiction.}
\end{proof}
\begin{proof}[Proof of \cref{thm:log}]
As mentioned in the paragraph before \cref{thm:log}, we can find the submanifold $C_p$ in $X\#(p-1)\overline{\CP}^2$ and $X_{(p)}$ is diffeomorphic to the rational blowdown of this $C_p$.
Since the fishtail class $\alpha$ satisfies the condition of \cref{Y:claim:simple:immersed}, the Bauer--Furuta basic classes of $X\#(p-1)\overline{\CP}^2$ have the form $L_J:=L+\sum J_iE_i$ where $L\in\calB(X)$ and $J_i=\pm1$ ($i=1,\cdots, p-1$).
Let $X_0$ be the complement of the interior of $C_p$.
Then $L+\left(\sum J_i\right) f\in H^2(X_{(p)};\Z)$  is a unique extension of $L_J|X_0$.
By \cref{thm:rbd}, the Bauer--Furuta basic classes of $X_{(p)}$ are of the form $L+\left(\sum J_i\right) f$  where $L\in\calB(X)$ and $J_i=\pm1$ ($i=1,\cdots, p-1$).
\end{proof}

\begin{lem}\label{Y:immersed:na}Suppose that a second homology class $\alpha$ of a connected oriented smooth 4-manifold $Z$   with $\alpha\cdot \alpha=0$ is represented by an immersed 2-sphere with exactly one positive double point. Then for any integer $n$, the class $n\alpha$ is represented by an immersed 2-sphere with exactly one positive double point. 
\end{lem}
\begin{proof}Since the 4-ball contains an immersed 2-sphere with exactly one positive double point, the $n=0$ case holds. We may thus assume $n>0$. 
	 A thickened neighborhood for the immersed 2-sphere representing $\alpha$ is a self-plumbing of a 2-sphere. Therefore, the neighborhood is the fishtail neighborhood given by the handlebody diagram shown in the first diagram of Figure~\ref{fig:slides} (see \cite[p.\ 240]{KiMe}, \cite[p.\ 202]{GS}). 
	In this figure, homology classes denote the classes given by corresponding 2-handles,  and we change the handle decomposition of the fishtail neighborhood by handle moves similar to those in \cite[Figures 12--17]{Yasui_AGT}, without changing the diffeomorphism type. We create a canceling pair of 2- and 3-handles and slide the resulting 2-handle over the 2-handle representng $\alpha$, as shown in the second diagram. We then slide 2-handles as indicated in the third and the fourth diagrams. By repeating these handle slides and using an isotopy, we obtain the fifth diagram. By sliding the 2-handle $n\alpha$ over the $-1$-framed 2-handle twice and the deleting the obvious canceling pair of 1- and 2-handles, we obtain the last diagram. We discuss this subhandlebody of the fishtail neighborhood. As shown in Figure~\ref{fig:immersed}, the attaching circle of the 2-hanldle representing $n\alpha$ becomes an unknot in the boundary of the 1-handlebody $S^1\times D^3$ by a positive crossing change. We thus see that the attaching circle bounds an immersed disk in the collar neighborhood of the boundary of the 1-handlebody.  Therefore, the class $n\alpha$ is represented by an immersed 2-sphere in $Z$ with exactly one positive double point. 
\begin{figure}[ht!]
\begin{center}
\includegraphics[width=4.5in]{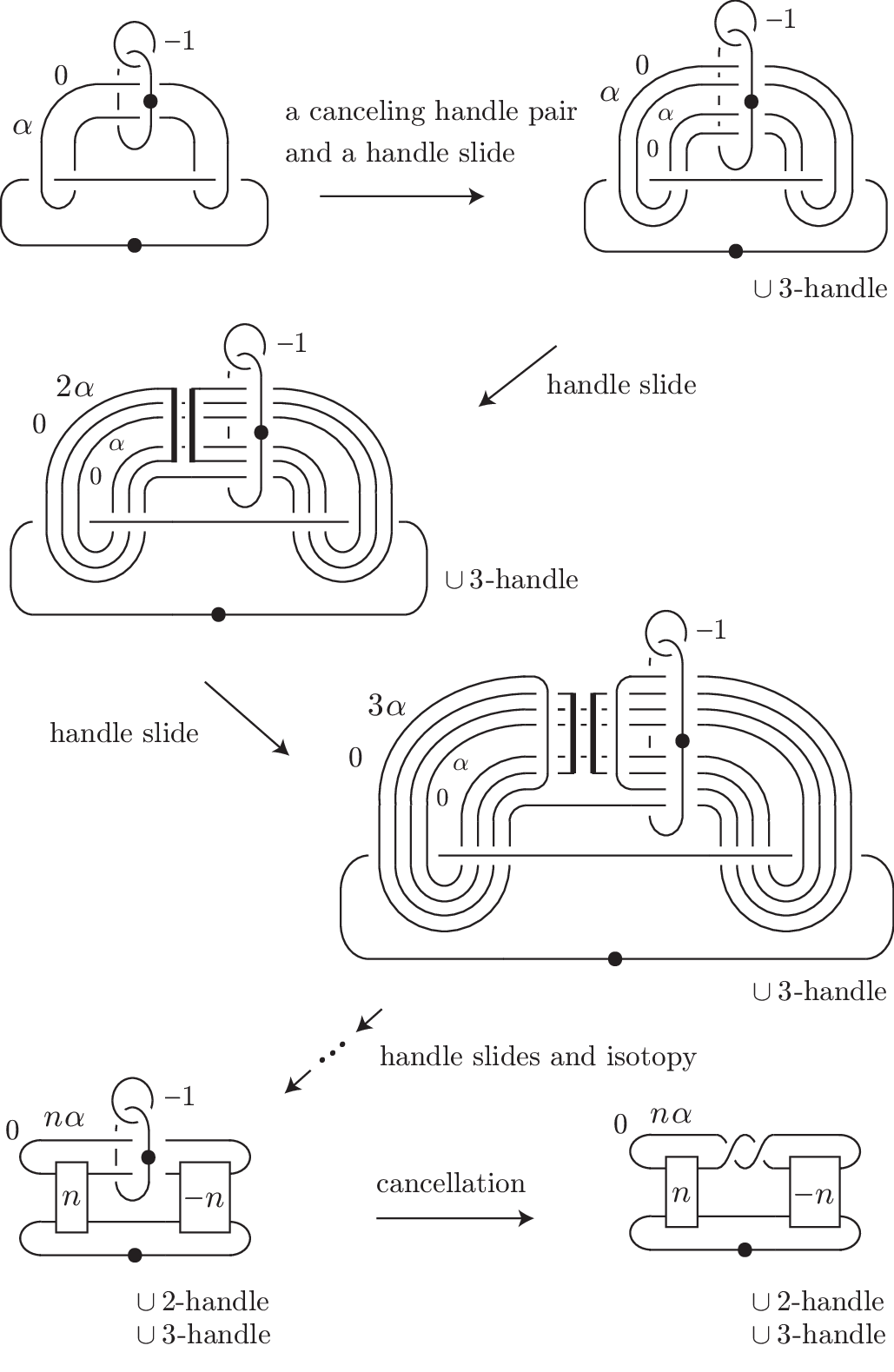}
\caption{Handle decompositions of a fishtail neighborhood}
\label{fig:slides}
\end{center}
\end{figure}
\begin{figure}[ht!]
\begin{center}
\includegraphics[width=3.9in]{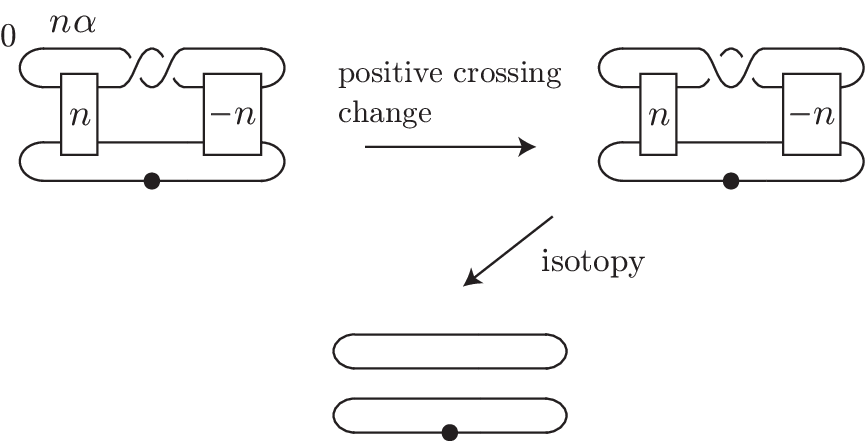}
\caption{Homotopy of the attaching circle}
\label{fig:immersed}
\end{center}
\end{figure}
\end{proof}
We remark that \cref{Y:claim:simple:embedded,Y:claim:simple:immersed} together with the connected sum formulae for Bauer--Furuta invariants produces many examples of smooth 4-manifolds of BF blowup simple type having the vanishig Seiberg--Witten invariant and yet a non-vanishing Bauer--Furuta invariant, which are not (at least trivially) produced from \cref{Y:claim:sum symplectic}. 

As an application of blowup simple type, we here give an alternative proof of the following theorem, which shows that all smooth structures are of mod 2 Seiberg--Witten simple type for a large class of topological 4-manifolds. This result was earlier obtained by  Kato, Nakamura and Yasui in \cite{KNY}. Although the essential idea is similar to the original proof, the proof simplifies the argument and especially does not use the adjunction inequality unlike the original proof. 

\begin{theo}[\cite{KNY}]\label{Y:claim:simple type conjecture}
Suppose that $X$ satisfies $b_2^+-b_1>1$ and $b_2^+-b_1\equiv 3\pmod{4}$, and let $\{\delta_1, \delta_2, \dots, \delta_k\}$ be a generating set of $H^1(X)$. If each cup product $\delta_i\cup \delta_j$ is either torsion or divisible by 2, then $X$ is of mod 2 Seiberg--Witten simple type. 
\end{theo}
\begin{proof}
To the contrary, suppose that $X$ is not of mod 2 Seiberg--Witten simple type.  Then, we see that $X$ has a mod 2 Seiberg--Witten basic class $K$ with $d_X(K)=2n$ for some positive integer $n$. By the blow up formula (\cite[Theorem 1.4]{FSimmersed}), $K+ 3E_1+\dots+3E_n$ is a mod 2 Seiberg--Witten basic class of $X\#n\overline{\CP}^2$. Let $X'$ be a closed connected oriented smooth 4-manifold with $b_2^+\equiv 3\pmod{4}$ and $b_1=0$ that admits a mod 2 Seiberg--Witten basic class $K'$ with $d_{X'}(K')=0$. We note that $K'$ is BF simple in $X'$ by \cref{Y:claim:SW simple=BF}. Also, there are many examples of $X'$ such as the $K3$ surface. By a result of Ishida and Sasahira \cite{IS}, we see that $K+ 3E_1+\dots+3E_n+K'$ is a Bauer--Furuta basic class of $X\#n\overline{\CP}^2\#X'$. 
\cref{Y:claim:non BF simple} thus shows that $K'$ is not BF simple, giving a contradiction.  
\end{proof}
As an application of \cref{thm:immersedadjunction}, we obtain the following immersed adjunction inequality.
\begin{theo}\label{Y:claim:adjunction_blowup}
Let $\alpha$ be a second homology class of $X$ represented by an immersed 2-sphere with $p$ positive double points and possibly with negative double points, where $p$ is a positive integer. If $X$ is of BF blowup simple type, then the following inequality holds for any Bauer--Furuta basic class $K$ of $X$. 
\begin{equation*}
\left|\langle K,\, \alpha \rangle \right| \,+\, \alpha\cdot\alpha \, \leq \, 2p-2. 
\end{equation*}
\end{theo}
\begin{proof}To the contrary, suppose that the above inequality does not hold.  We may assume that the inequality 
\begin{equation*}
\langle K,\, \alpha \rangle \,+\, \alpha\cdot\alpha \, > \, 2p-2 
\end{equation*}
holds, if necessary by replacing $\alpha$ with $-\alpha$. Furthermore, if necessary by resolving the double points using blowing ups, we may assume that $\alpha-2e_1$ is represented by an immersed 2-sphere having only $p-1$ positive double points in $X\#k\overline{\CP}^2$ for some positive integer $k$. 
In $X\#k\overline{\CP}^2$, the inequality 
\begin{equation*}
\langle K+E_1+\dots+E_k,\, \alpha-2e_1 \rangle \,+\, (\alpha-2e_1)\cdot(\alpha-2e_1) \, > \, 2(p-1)-2. 
\end{equation*}
holds, since the left hand slide is $\langle K,\, \alpha \rangle + \alpha\cdot \alpha-2$. 
\cref{thm:immersedadjunction} thus shows that $K+E_1+\dots+E_k+2PD(\alpha-2e_1)$ is a Bauer--Furuta basic class of $X\#\overline{\CP}^2$. Since this class is equal to $K+2PD(\alpha)-3E_1+E_2+\dots+E_{k}$, it follows from \cref{thm:blowup} that  $K+2PD(\alpha)$ and $K+2PD(\alpha)-3E_1$ are Bauer--Furuta basic classes of respectively $X$ and $X\#\overline{\CP}^2$. Hence, $X$ is not of BF blowup simple type, giving a contradiction.
\end{proof}

\begin{rem}
Bauer \cite[\S4]{B_survey} introduced a more refined version of the Bauer--Furuta invariant.
By using this refined invariant, Bauer extended \cref{thm:BF-SW} which gives the relation between the Seiberg--Witten invariants and the Bauer--Furuta invariants  to the case when $b_2^+\geq 2$.
Furthermore, as mentioned in \cite[\S8]{B_survey}, the connected-sum formula also holds for this version of the invariants.
Then several results in this section is generalized to the case when $b_2^+\geq 2$, if we define the notions of Bauer--Furuta basic class, BF simple and BF blowup simple type by using the refined version of the Bauer--Furuta invariant.
In fact, \cref{Y:claim:SW simple=BF}, \cref{Y:claim:symplectic}  and \cref{Y:claim:sum symplectic} are extended to the  the case when $b_2^+\geq 2$.
\end{rem}

  \section{The stable cohomotopy groups of complex projective spaces}\label{cohomotopy}

\subsection{Preliminaries}
  \subsubsection{Long exact sequence}

  Recall that if $A$ is a subcomplex of a pointed CW complex $X$, then for any pointed space $Y$, there is an exact sequence of pointed sets
  \[
    [X/A,Y]\to[X,Y]\to[A,Y].
  \]
  Then since direct limits preserve exact sequences, we get a long exact sequence 
  \[
    \cdots\to\{\Sigma(X/A),Y\}\to\{\Sigma X,Y\}\to\{\Sigma A,Y\}\to\{X/A,Y\}\to\{X,Y\}\to\{A,Y\}\to\cdots,
  \]
  where $\{Z,W\}$ denotes the stable homotopy set $\displaystyle\lim_{n\to\infty}[\Sigma^n,Z,\Sigma^nW]$. On the other hand, by the Freudenthal suspension theorem and the Blakers--Massey theorem, there is also an exact sequence
  \[
    \cdots\{Z,A\}\to\{Z,X\}\to\{Z,X/A\}\to\{Z,\Sigma A\}\to\{Z,\Sigma X\}\to\{Z,\Sigma(X/A)\}\to\cdots.
  \]


  \subsubsection{Stable homotopy groups of spheres}

  Let $\pi_n^S$ denote the $n$-th stable homotopy group of $S^0$. Then as in \cite{T}, we have $\pi^S_*=0$ for $*<0$ and
  \[
    \pi_0^S=\Z\{\iota\},\quad\pi_1^S=\Z/2\{\eta\},\quad\pi_2^S=\Z/2\{\eta^2\},\quad\pi_3^3=\Z/24\{\nu\},\quad\pi_4^S=\pi_5^S=0,
  \]
  where $\Z/n\{a\}$ means a cyclic group $\Z/n$ with generator $a$. Moreover, we have
  \[
    \eta^3=12\nu.
  \]


  \subsection{Stunted projective spaces}

  Let $\C P^n_k=\C P^n/\C P^{n-k}$ and $\H P^n_k=\H P^n/\H P^{n-k}$. Then
  \[
    \C P^n_k=S^{2(n-k+1)}\cup e^{2(n-k+2)}\cup\cdots\cup e^{2n}\quad\text{and}\quad\H P^n_k=S^{4(n-k+1)}\cup e^{4(n-k+2)}\cup\cdots\cup e^{4n}.
  \]

  \begin{lem}
    \label{CP}
    For $n\geq 3$, we have
    \[
     \C P^n_{n-2} = S^{2n-2}\cup_{(n-1)\eta}e^{2n}.
    \]
  \end{lem}

  \begin{proof}
    Look at the action of $\mathrm{Sq}^2$ on the mod 2 cohomology of $\C P^n_{n-2}$.
  \end{proof}

  In \cite{J}, quasi-projective spaces are defined, and in the quaternionic case, we have
  \[
    Q^n=S^3\cup e^7\cup\cdots\cup e^{4n-1}.
  \]
  Let $Q^n_k=Q^n/Q^{n-k}$. The following is proved by James \cite{J}.

  \begin{lem}
    \label{HP-Q}
    For $n\geq 3$, we have
    \[
      \H P^n_{n-2}=S^{4n-4}\cup_{(n-1)\nu}e^{4n}\quad\text{and}\quad Q^n_{n-2}=S^{4n-5}\cup_{n\nu}e^{4n-1}.
    \]
  \end{lem}

For $n\geq 4$,  consider the cell decompositions
  \[
    \C P^n_{n-3}=S^{2n-4}\cup_\alpha e^{2n-2}\cup_\beta e^{2n}\quad\text{and}\quad \C P^n_{n-2}=S^{2n-2}\cup_{\bar{\beta}}e^{2n},
  \]
  where $\bar{\beta}$ is the composite of maps \[
  S^{2n-1}\xrightarrow{\beta}S^{2n-4}\cup_\alpha e^{2n-2}\to (S^{2n-4}\cup_\alpha e^{2n-2})/S^{2n-4}=S^{2n-2}.
  \]  
  Then by Lemma \ref{CP}, $\alpha=(n-2)\eta$ and $\bar{\beta}=(n-1)\eta$. Thus for $n$ odd,
  \[
    \C P^n_{n-3}=S^{2n-4}\cup_\eta e^{2n-2}\cup_{d_n\nu} e^{2n}
  \]
  and for $n$ even,
  \[
    \C P^n_{n-3}=(S^{2n-4}\vee S^{2n-2})\cup_{d_n\nu+\eta} e^{2n},
  \]
for some integer $d_n$.
  We determine the integer $d_n$. For degree reasons, $\Sigma\C P^n_{n-3}$ is the $(2n+1)$-skeleton of the Stiefel manifold $U(n+1)/U(n-2)$, and $Q^n_{n-3}$ is the $(4n-1)$-skeleton of $Sp(n)/Sp(n-3)$. The natural inclusion $Sp(n)\to U(2n)$ induces a map $Sp(n)/Sp(n-3)\to U(2n)/U(2n-6)$, and so by the cellular approximation theorem, there is a dotted arrow $f$ in the homotopy commutative diagram
  \[
    \xymatrix{
      Q^n_{n-3}\ar@{-->}[r]^f\ar[d]&\Sigma\C P^{2n-1}_{2n-4}\ar[d]\\
      Sp(n)/Sp(n-3)\ar[r]&U(2n)/U(2n-3)
    }
  \]
  such that $f$ is identified with the inclusion
  \[
    S^{4n-5}\cup_{d_{2n-1}\nu} e^{4n-1}\to S^{4n-5}\cup_\eta e^{4n-3}\cup_{d_{2n-1}\nu} e^{4n-1}.
  \]
  Then by Lemma \ref{HP-Q}, $d_{2n-1}=n$. On the other hand, the natural map $\C P^{2n}_{2n-3}\to\H P^n_{n-2}$ is identified with the pinch map
  \[
    (S^{4n-4}\vee S^{4n-2})\cup_{d_{2n}\nu+\eta} e^{4n}\to S^{4n-4}\cup_{d_{2n}\nu} e^{4n}
  \]
  and so by Lemma \ref{HP-Q}, $d_{2n}=n-1$. Summarizing, we obtain:

  \begin{prop}
    \label{cell decomposition}
    Let $n\geq 4$.
    For $n$ odd, 
    \[
      \C P^n_{n-3}=S^{2n-4}\cup_\eta e^{2n-2}\cup_{\frac{n+1}{2}\nu}e^{2n}
    \]
    and for $n$ even, 
    \[
      \C P^n_{n-3}=(S^{2n-4}\vee S^{2n-2})\cup_{\frac{n-2}{2}\nu+\eta}e^{2n}.
    \]
  \end{prop}


  \subsection{Stable cohomotopy groups}

  Let $\pi^*(X)$ denote the stable cohomotopy group. Then there is a natural transformation
  \[
    h^*\colon\pi^*(X)\to H^*(X)
  \]
  called the Hurewicz homomorphism. We determine the kernel and the cokernel of the Hurewicz homomorphism $h^{2n-k}\colon\pi^{2n-k}(\C P^n)\to H^{2n-k}(\C P^n)$ for $0\le k\le 6$.

  \begin{prop}
    \label{k=0}
    The Hurewicz homomorphism $h^{2n}\colon\pi^{2n}(\C P^n)\to H^{2n}(\C P^n)$ is an isomorphism.
  \end{prop}

  \begin{proof}
    There is an exact sequence
    \[
      \pi^{2n-1}(\C P^{n-1})\to\pi^{2n}(\C P^n_{n-1})\to\pi^{2n}(\C P^n)\to\pi^{2n}(\C P^{n-1}).
    \]
    Clearly, $\pi^*(\C P^{n-1})=0$ for $*=2n-1,2n$, so the map $\pi^{2n}(\C P^n_{n-1})\to\pi^{2n}(\C P^n)$ is an isomorphism. There is also a commutative diagram
    \[
      \xymatrix{
        \pi^{2n}(\C P^n_{n-1})\ar[r]^\cong\ar[d]_{h^{2n}}&\pi^{2n}(\C P^n)\ar[d]^{h^{2n}}\\
        H^{2n}(\C P^n_{n-1})\ar[r]^\cong&H^{2n}(\C P^n).
      }
    \]
    Since the left vertical map is an isomorphism, the proof is finished.
  \end{proof}

  \begin{prop}
    \label{k=1}
    $\mathrm{Ker}\,h^{2n-1}\cong\Z/(2,n+1)$ and $\mathrm{Coker}\,h^{2n-1}=0$.
  \end{prop}

  \begin{proof}
    Since $H^{2n-1}(\C P^n)=0$, $\mathrm{Ker}\,h^{2n-1}=\pi^{2n-1}(\C P^n)$ and $\mathrm{Coker}\,h^{2n-1}=0$. Arguing as in the proof of Proposition \ref{k=0},
    \[
      \pi^{2n-1}(\C P^n)\cong\pi^{2n-1}(\C P^n_{n-2}).
    \]
    For $n$ odd, $\C P^n_{n-2}=S^{2n-2}\vee S^{2n}$ by Lemma \ref{CP}, and so $\pi^{2n-1}(\C P^n_{n-2})\cong\pi^S_{-1}\oplus\pi^S_1\cong\Z/2$. For $n$ even, $\C P^n_{n-2}=S^{2n-2}\cup_\eta e^{2n}$ by Lemma \ref{CP}, so that there is an exact sequence
    \[
      \pi_0^S\xrightarrow{\eta^*}\pi^S_1\to\pi^{2n-1}(\C P^n_{n-2})\to\pi^S_{-1}.
    \]
    Then $\pi^{2n-1}(\C P^n_{n-2})=0$, completing the proof.
  \end{proof}

  \begin{prop}
    \label{k=2}
    $\mathrm{Ker}\,h^{2n-2}\cong\Z/(2,n-1)$ and $\mathrm{Coker}\,h^{2n-2}=\Z/(2,n)$.
  \end{prop}

  \begin{proof}
    As in the proof of Proposition \ref{k=0}, we can see that the map $\pi^{2n-2}(\C P^n_{n-2})\to\pi^{2n-2}(\C P^n)$ is an isomorphism. Consider the inclusion of the bottom cell $S^{2n-2}\to\C P^n_{n-2}$. Then we get a commutative diagram
    \[
      \xymatrix{
        \pi^{2n-2}(S^{2n-2})\ar[d]^{h^{2n-2}}_\cong&\pi^{2n-2}(\C P^n_{n-2})\ar[r]^\cong\ar[l]\ar[d]^{h^{2n-2}}&\pi^{2n-2}(\C P^n)\ar[d]^{h^{2n-2}}\\
        H^{2n-2}(S^{2n-2})&H^{2n-2}(\C P^n_{n-2})\ar[r]^\cong\ar[l]_\cong&H^{2n-2}(\C P^n).
      }
    \]
    Then it is sufficient to compute the kernel and the cokernel of the map $\pi^{2n-2}(\C P^n_{n-2})\to\pi^{2n-2}(S^{2n-2})=\pi^S_0$. By Lemma \ref{CP}, there is an exact sequence
    \[
      \pi^S_1\xrightarrow{(n-1)\eta^*}\pi^S_2\to\pi^{2n-2}(\C P^n_{n-2})\to\pi^S_0\xrightarrow{(n-1)\eta^*}\pi^S_1.
    \]
    Thus the proof is complete.
  \end{proof}

  \begin{prop}
    \label{k=3}
    $\mathrm{Coker}\,h^{2n-3}=0$ and
    \[
      \mathrm{Ker}\,h^{2n-3}\cong
      \begin{cases}
        \Z/(24,n+1)&n\text{ odd}\\
        \Z/\frac{(24,n-2)}{2}&n\text{ even}.
      \end{cases}
    \]
  \end{prop}

  \begin{proof}
    As in the proof of Proposition \ref{k=1}, we can see that $\mathrm{Coker}\,h^{2n-3}=0$ and $\mathrm{Ker}\,h^{2n-3}\cong\pi^{2n-3}(\C P^n_{n-3})$. For $n$ odd, it follows from Proposition \ref{cell decomposition} that there is an exact sequence
    \[
      \pi_0^S\xrightarrow{\eta^*+\frac{n+1}{2}\nu^*}\pi_1^S\oplus\pi^S_3\to\pi^{2n-3}(\C P^n_{n-3})\to\pi^S_{-1}.
    \]
    Then $\pi^{2n-3}(\C P^n_{n-3})\cong(\Z/2\{\eta\}\oplus\Z/24\{\nu\})/(\eta+\frac{n+1}{2}\nu)\cong\Z/(24,n+1)$. For $n$ even, it follows from Proposition \ref{cell decomposition} that there is an exact sequence
    \[
      \pi^S_0\oplus\pi^S_2\xrightarrow{\frac{n-2}{2}\nu^*\oplus\eta^*}\pi^S_3\to\pi^{2n-3}(\C P^n_{n-3})\to\pi^S_{-1}\oplus\pi^S_1\xrightarrow{\frac{n-2}{2}\nu^*\oplus\eta^*}\pi^S_2.
    \]
    Then $\pi^{2n-3}(\C P^n_{n-3})\cong(\Z/24\{\nu\})/(\frac{n-2}{2},12)\nu\cong\Z/\frac{(24,n-2)}{2}$, completing the proof.
  \end{proof}

  \begin{prop}
    \label{k=4}
    $\mathrm{Ker}\,h^{2n-4}=0$ and
    \[
      \mathrm{Coker}\,h^{2n-4}\cong
      \begin{cases}
        \Z/\frac{48}{(24,n+1)}&n\text{ odd}\\
        \Z/\frac{24}{(24,n-2)}&n\text{ even}.
      \end{cases}
    \]
  \end{prop}

  \begin{proof}
    As in the proof of Proposition \ref{k=2}, we can see that the Hurewicz map $h^{2n-4}\colon\pi^{2n-4}(\C P^n)\to H^{2n-4}(\C P^n)$ is identified with the map $\pi^{2n-4}(\C P^n)\to\pi^S_0$ induced from the bottom cell inclusion. For $n$ odd, by Lemma \ref{cell decomposition}, there is an exact sequence
    \[
      \pi^S_1\xrightarrow{\eta^*+\frac{n+1}{2}\nu^*}\pi_2^S\oplus\pi_4^S\to\pi^{2n-4}(\C P^n_{n-3})\to\pi_0^S\xrightarrow{\eta^*+\frac{n+1}{2}\nu^*}\pi_1^S\oplus\pi_3^S.
    \]
    Then by an easy inspection, we can see that the kernel and the cokernel of the map $\pi^{2n-4}(\C P^n_{n-3})\to\pi_0^S$ are as in the statement. For $n$ even, by Lemma \ref{cell decomposition}, there is an exact sequence
    \[
      \pi_4^S\to\pi^{2n-4}(\C P^n_{n-3})\to\pi_0^S\oplus\pi_2^S\xrightarrow{\frac{n-2}{2}\nu^*+\eta^*}\pi_3^S.
    \]
    Thus we can see that the kernel and the cokernel of the map $\pi^{2n-4}(\C P^n_{n-3})\to\pi_0^S$ are as in the statement too, completing the proof.
  \end{proof}

  \begin{prop}
    \label{k=5}
    $\mathrm{Coker}\,h^{2n-5}=0$ and
    \[
      \mathrm{Ker}\,h^{2n-5}\cong
      \begin{cases}
        \Z/(24,n)&n\text{ even}\\
        \Z/\frac{(24,n-3)}{2}&n\text{ odd}.
      \end{cases}
    \]
  \end{prop}

  \begin{proof}
    As in the proof of Proposition \ref{k=3}, $\mathrm{Coker}\,h^{2n-5}=0$ and $\mathrm{Ker}\,h^{2n-5}\cong\pi^{2n-5}(\C P^n_{n-4})$. Since there is an exact sequence
    \[
      0=\pi^S_5=\pi^{2n-5}(S^{2n})\to\pi^{2n-5}(\C P^n_{n-4})\to\pi^{2n-5}(\C P^{n-1}_{n-4})\to\pi^{2n-5}(S^{2n-1})=\pi^S_4=0,
    \]
    the map $\pi^{2n-5}(\C P^n_{n-4})\to\pi^{2n-5}(\C P^{n-1}_{n-4})$ is an isomorphism. Then by Proposition \ref{k=3}, the proof is done.
  \end{proof}

  \begin{prop}
    \label{k=6}
    $\mathrm{Ker}\,h^{2n-6}=0$ and
    \[
      \mathrm{Coker}\,h^{2n-6}\cong
      \begin{cases}
        \Z/\frac{48}{(24,n)}&n\text{ even}\\
        \Z/\frac{24}{(24,n-3)}&n\text{ odd}.
      \end{cases}
    \]
  \end{prop}

  \begin{proof}
    As in the proof of Proposition \ref{k=2}, the Hurewicz map $h^{2n-6}\colon\pi^{2n-6}(\C P^n)\to H^{2n-6}(\C P^n)$ is identified with the map $\pi^{2n-6}(\C P^n_{n-4})\to\pi_0^S$ induced from the bottom cell inclusion. Consider an exact sequence
    \[
      \pi^{2n-7}(\C P^{n-1}_{n-4})\to\pi_6^S\to\pi^{2n-6}(\C P^n_{n-4})\to\pi^{2n-6}(\C P^{n-1}_{n-4})\to\pi^S_5=0.
    \]
    Since $\eta\nu=0$, there is a map $\bar{\nu}\colon S^{2n-4}\cup_\eta e^{2n-2}\to S^{2n-7}$ which restricts to $\nu$ on the bottom cell. For $n$ odd, by Proposition \ref{cell decomposition}, we can consider the composition of maps
    \[
      S^{2n-1}\xrightarrow{\nu}S^{2n-4}\to\C P^{n-1}_{n-4}\to S^{2n-4}\cup_\eta e^{2n-2}\xrightarrow{\bar{\nu}}S^{2n-7}
    \]
    which coincides with $\nu^*$. Then the map $\pi^{2n-6}(\C P^n_{n-4})\to\pi^{2n-6}(\C P^{n-1}_{n-4})$ is an isomorphism. For $n$ even, by Proposition \ref{cell decomposition}, there is an exact sequence
    \[
      0=\pi_5^S\to\pi_{2n-1}(\C P^{n-1}_{n-4})\to\pi_3^S\oplus\pi_1^S\to\pi_4^S=0,
    \]
    implying that there is a map $S^{2n-1}\to\C P^{n-1}_{n-4}$ such that the composition $S^{2n-1}\to\C P^{n-1}_{n-4}\to S^{2n-4}$ coincides with $\nu$. Then we obtain that the map $\pi^{2n-6}(\C P^n_{n-4})\to\pi^{2n-6}(\C P^{n-1}_{n-4})$ is an isomorphism, so there is a commutative diagram
    \[
      \xymatrix{
        \pi^{2n-6}(\C P^n_{n-4})\ar[r]^\cong\ar[d]&\pi^{2n-6}(\C P^{n-1}_{n-4})\ar[d]\\
        \pi_0^S\ar@{=}[r]&\pi_0^S.
      }
    \]
    Thus the proof is done by Proposition \ref{k=4}.
  \end{proof}


  \subsection{Naturality}\label{naturality}

  We consider the map $\pi^{2n-i}(\C P^n)\to\pi^{2n-i}(\C P^{n-j})$ induced from the inclusion $\C P^{n-j}\to\C P^n$ for $i=5,6$ and $j=1,2$.

  \begin{prop}
    \label{naturality k=5,6}
    The map $\pi^{2n-i}(\C P^n)\to\pi^{2n-i}(\C P^{n-1})$ is trivial for $i=3$, and an isomorphism for $i=5,6$.
  \end{prop}

  \begin{proof}
  To prove the case $i=3$,  we only need to show this map is trivial when $n$ is odd  by Proposition \ref{k=1}. 
 There is an exact sequence
    \[
      \pi^{2n-3}(\C P^{n})\to\pi^{2n-3}(\C P^{n-1})\to\pi^S_2
    \]
    such that as in the proof of Proposition \ref{k=1}, the last map is identified with the map
    \[
      \eta^*\colon\pi^S_1\to\pi^S_2
    \]
    which is non-trivial, so an isomorphism.
    The case $i=5$ is proved in the proof of Proposition \ref{k=5}. By Propositions \ref{k=4} and \ref{k=6}, there is a commutative diagram
    \[
      \xymatrix{
        0\ar[r]&\pi^{2n-6}(\C P^n)\ar[r]\ar[d]&H^{2n-6}(\C P^n)\ar[d]^\cong\ar[r]&A\ar[r]\ar[d]&0\\
        0\ar[r]&\pi^{2n-6}(\C P^{n-1})\ar[r]&H^{2n-6}(\C P^{n-1})\ar[r]&A\ar[r]&0
      }
    \]
    for some finite cyclic group $A$.
 Then the map $A\to A$ in the diagram turns out to be an isomorphism, and so the map $\pi^{2n-6}(\C P^n)\to\pi^{2n-6}(\C P^{n-1})$ is an isomorphism too, completing the proof.
  \end{proof}

  \begin{prop}
    The map $\pi^{2n-5}(\C P^n)\to\pi^{2n-5}(\C P^{n-2})$ is trivial.
  \end{prop}

  \begin{proof}
  The map is the composition
  \[
  \pi^{2n-5}(\C P^n)\to\pi^{2n-5}(\C P^{n-1})\to\pi^{2n-5}(\C P^{n-2}).
  \]
  The latter map is trivial by Proposition \ref{naturality k=5,6}.
  \end{proof}

  \begin{prop}
    The map $\pi^{2n-6}(\C P^n)\to\pi^{2n-6}(\C P^{n-2})$ is identified with the map
    \[
      \Z\to\Z,\quad x\mapsto \frac{24}{(24,n)}x
    \]
    for $n$ even and
    \[
      \Z\to\Z\oplus\Z/2,\quad x\mapsto\left(\frac{24}{(24,n-3)}x,0\right)
    \]
    for $n$ odd.
  \end{prop}

  \begin{proof}
    The proof is quite the same as that of Proposition \ref{naturality k=5,6} for $i=6$.
  \end{proof}

\end{document}